%% file: CCSmallBridgeNumber.tex
\documentclass[10pt]{article}

\pdfoutput=1
\usepackage{graphicx}
\usepackage{amsmath}
\usepackage{amssymb}
\usepackage{amsthm}
\usepackage[margin=1.0in]{geometry}
\usepackage{setspace}
\usepackage{verbatim}
\usepackage[backend=bibtex]{biblatex}
\usepackage{enumerate}
\usepackage{subfig}

\newcommand{\bdd}{\mbox{$\partial$}}
\newcommand{\define}[1]{\emph{\textbf{#1}}}
\newcommand{\redG}{G_\mathcal{S}}
\newcommand{\lvlG}{G_{\mathcal{T}}}
\newcommand{\redS}{\mathcal{S}}
\newcommand{\lvlS}{\mathcal{T}}

\newtheorem{thm}{Theorem}[section]

\newtheorem{defn}[thm]{Definition}
\newtheorem{lem}[thm]{Lemma}
\newtheorem{prop}[thm]{Proposition}
\newtheorem{cor}[thm]{Corollary}
\newtheorem{claim}[thm]{Claim}
\newtheorem{rem}[thm]{Remark}
\newtheorem{conj}{Conjecture}

\begin{document}
\title{Cabling Conjecture for Small Bridge Number}
\date{\today}
\author{Colin Grove}
\maketitle

\input{sections/PaperAbstract.tex}
\input{sections/PaperIntro.tex}
\input{sections/Setup.tex}
\input{sections/IntersectionGraphs.tex}

\input{sections/GLProofRewrite.tex}
\input{sections/RotationFreeGraph.tex}

\input{sections/CablingConjecture.tex}

\printbibliography

\end{document}

%% file: sections/PaperAbstract.tex
\begin{abstract}
Let $k\subset S^3$ be a nontrivial knot. The Cabling Conjecture of
Francisco Gonz\'alez-Acu\~na and Hamish Short \cite{GonzalesAcunaShort86}
posits that $\pi$-Dehn
surgery on $k$ produces a reducible manifold if and only if $k$ is a
$(p,q)$-cable knot and the surgery slope $\pi$ equals $pq$.
We extend the work of James Allen Hoffman
\cite{Hoffman95} to prove the Cabling Conjecture for knots with
bridge number up to $5$.
\end{abstract}

%% file: sections/PaperIntro.tex
\section{Introduction}
Let $k\subset S^3$ be a knot, and let $N(k)\subset S^3$ be the open
neighborhood of $k$. A \define{slope} is an isotopy class of nontrivial
simple closed curves on $\bdd \overline{N(k)}$. If $\pi$ is a slope, 
\define{$\pi$-Dehn surgery} on $k$ consists of drilling out $N(k)$, and
gluing a solid torus $T$ to $M = S^3 \setminus N(k)$ such that a meridian
curve on $\bdd T$ is mapped to $\pi$.

The usefulness of Dehn surgery in the study of $3$-manifolds is established
by the following theorem.
\begin{thm}[Lickorish-Wallace, \cite{Wallace60} and \cite{Lickorish62}]
 Any closed, orientable, connected $3$-manifold can be obtained from $S^3$
 by a finite collection of Dehn surgeries.
\end{thm}

If $M$ is a $3$-manifold such that every embedded $2$-sphere in $M$ bounds
an embedded $3$-ball, $M$ is \define{irreducible}. An embedded $2$-sphere
in $M$ that bounds no $3$-ball is called a \define{essential} (or
a \define{reducing sphere}), and
the existence of an essential sphere makes $M$ \define{reducible}.
One naturally wonders about the relationship between these two important
tools, i.e., when does Dehn surgery
produce a reducible manifold?

Suppose $k\subset S^3$ is the unknot, $m$ the meridian slope on
$\bdd \overline{N(k)}$,
and $l$ the longitudinal slope (i.e. the slope which bounds a disk
in $S^3 \setminus N(k)$). Then $m$-Dehn surgery produces $S^3$,
$l$-Dehn surgery produces $S^2 \times S^1$ (see Figure 
\ref{fig:TrivReducible}), and Dehn surgery with any
other slope produces a lens space.
 \begin{figure}[h]
 \centering
  \subfloat[][$D_1\subset T$.]{\includegraphics[width=0.4\textwidth]{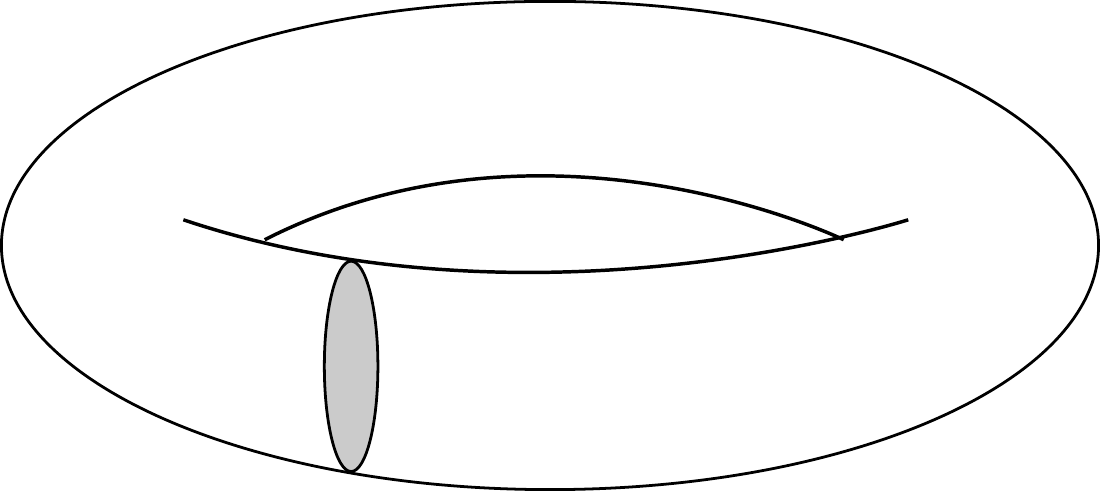}
                 \label{fig:TrivDiskInside}}
  \hspace{1cm}
  \subfloat[][$D_2\subset S^3 \setminus N(unknot)$.]{\includegraphics[width=0.4\textwidth]{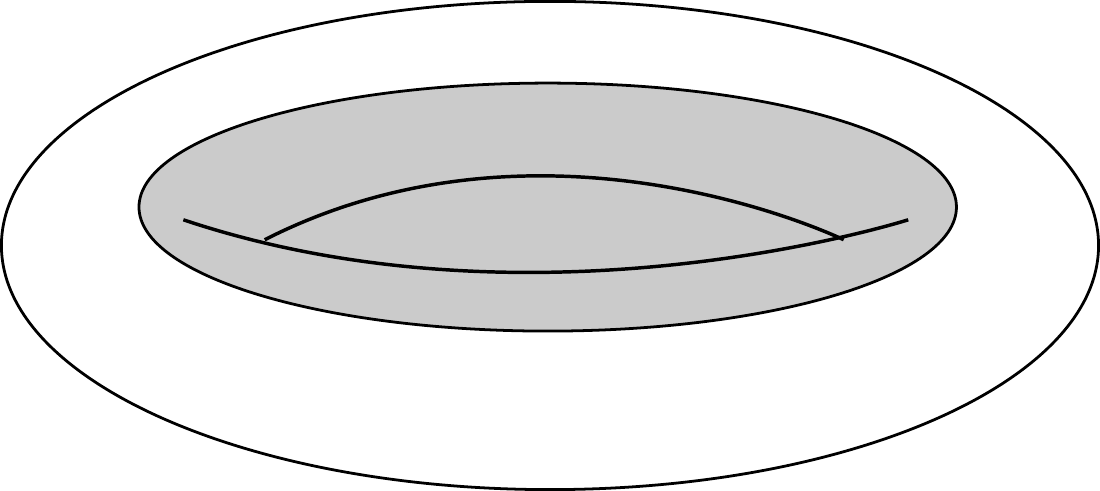}
                 \label{fig:TrivDiskOutside}}
  \caption[The essential sphere obtained from longitudinal surgery on the unknot.]{After longitudinal surgery, $D_1\cup D_2$ is an essential sphere.}
  \label{fig:TrivReducible}
\end{figure}

The Cabling Conjecture addresses when surgery on a nontrivial knot
produces a reducible manifold.
It should first be noted that David Gabai showed that only separating
 spheres need be considered, by showing that surgery on a nontrivial
 knot produces
neither $S^2 \times S^1$ (thereby proving the Property R Conjecture)
nor any manifold with an $S^2 \times S^1$ summand (proving the Poenaru
Conjecture) \cite{Gabai87}.

The Cabling Conjecture (presented in 1983) makes a claim about the specific
circumstances under which Dehn surgery on a nontrivial $k\subset S^3$
produces a reducible manifold.
\begin{conj}[Gonz\'ales-Acu\~na and Short, \cite{GonzalesAcunaShort86}]
 Let $k\subset S^3$ be a nontrivial knot. Then $\pi$-Dehn surgery on $k$
 produces a reducible manifold if and only if $k$ is a $(p,q)$-cable knot
 and the surgery slope $\pi$ equals $pq$.
\end{conj}

In order to understand the Cabling Conjecture we must define cable knots.
A nontrivial knot $k\subset S^3$ is a \define{$(p,q)$-torus knot} if $k$
can be isotoped to a $\frac{p}{q}$-curve in the boundary of an
unknotted solid torus $T\subset S^3$.
Let $e:T \to S^3$ be an embedding of $T$ into $S^3$. If $k\subset \bdd T$ is
a $(p,q)$-torus knot with $p > 1$, then $e(k)$ is a
\define{$(p,q)$-cable knot}.

One direction of the Cabling Conjecture is known: a $pq$-Dehn surgery on
a $(p,q)$-cable knot produces a reducible manifold. To see this, let $T$
be the solid (possibly knotted) torus on which the $(p,q)$-cable knot
$k$ lies. Cutting $\bdd T$ along $k$ produces an annulus $A$, and
both components of $\bdd A$ represent the same slope on
$\bdd \overline{N(k)}$. It follows that $\pi$-Dehn surgery, with slope
$\pi$ equal to the slope represented by $\bdd A$, will produce a reducible
manifold, since meridian disks $D_1$, $D_2$ of the filling torus $V$ can
be found such that $\bdd D_1$ and $\bdd D_2$ are precisely $\bdd A$.
Thus $A \cup D_1 \cup D_2$ is a sphere, and essential since none of
$D_1$, $D_2$, or $A$ are boundary parallel to $\bdd V$. $A$ is called the
\define{cabling annulus}.
 \begin{figure}[!h]
  \centering
  \includegraphics[width=0.5\textwidth]{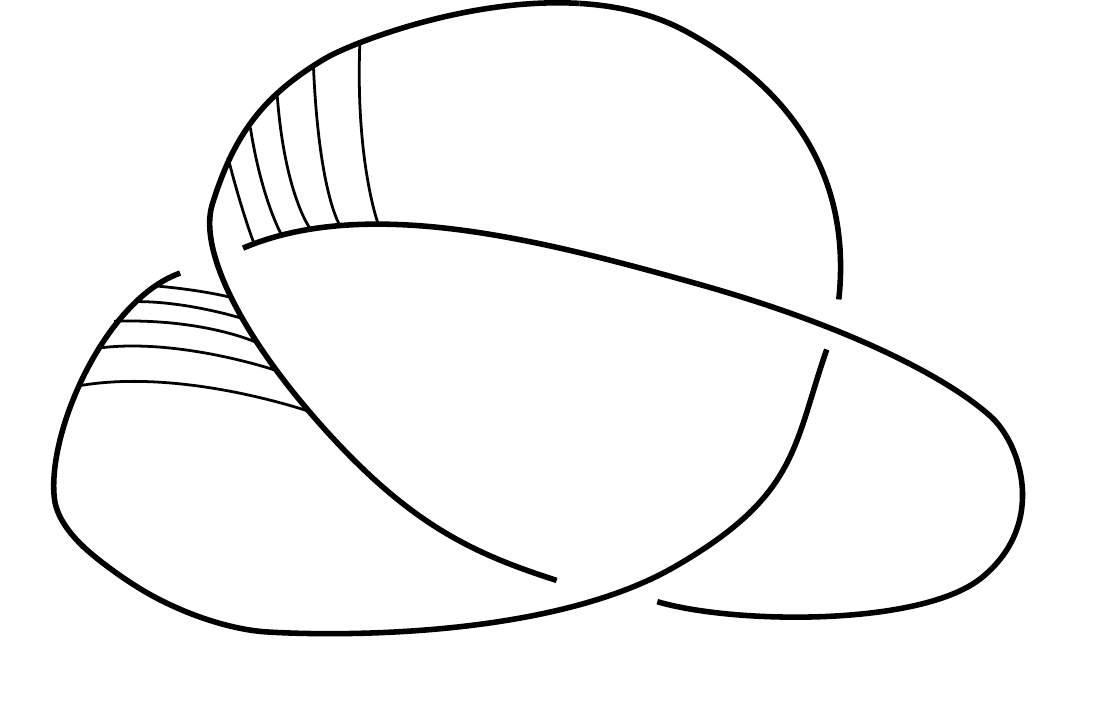}
  \caption{A section of the cabling annulus on the trefoil knot.}
  \label{fig:CablingAnnulus}
 \end{figure}

To see that the appropriate slope $\pi$ is $pq$, we first note that
since the boundary of $A$ on $\bdd T$ is $k$, each component of $\bdd A$
traverses $k$ along a longitude once, so each component of $\bdd A$ has
integral slope on $\bdd \overline{N(k)}$. Each time $k$ goes once around
the longitude of $T$, a $\bdd A$ component goes $q$ times around the
meridian of $\bdd \overline{N(k)}$. Since $k$ goes $p$ times around the
longitude of $T$, the slope of $\bdd A$ on $\bdd \overline{N(k)}$ is $pq$.

The Cabling Conjecture has been proven for many classes of knots, including:
\begin{itemize}
  \setlength\itemsep{.1em}
\item composite knots by Gordon in 1983 \cite{Gordon83};
\item satellite knots by Scharlemann in 1989 \cite{Scharlemann90};
\item strongly invertible knots by Eudave-Mu\~{n}oz in 1992
      \cite{EudaveMunoz92};
\item alternating knots by Menasco and Thistlethwaite in 1992
      \cite{MenascoThistlethwaite92};
\item arborescent knots by Wu in 1994 \cite{Wu96};
\item knots of bridge number up to $4$ in 1995 \cite{Hoffman95};
\item symmetric knots by Hayashi and Shimokawa in 1998 \cite{HayashiShimokawa98};
\item knots of bridge number at least $6$ and distance at least $3$,
      by Blair, Campisi, Johnson, Taylor and Tomova in 2012 \cite{BCJTT12}.
\end{itemize}

Due to the results of Hoffman \cite{Hoffman95} and Blair et. al.
\cite{BCJTT12}, proving the Cabling Conjecture for $5$-bridge knots
restricts remaining cases to knots with low distance.

We assume the existence of a reducing surgery on a knot that does not
satisfy the Cabling Conjecture, and find a planar
surface in the knot exterior such that its intersection with the
reducing sphere has certain desirable properties. The arcs of intersection
are treated as edges in graphs on the planar surface and the reducing
sphere, and we use combinatorial methods to show the existence of various
structure in each graph. The structure we find, combined with recent
results of Zufelt \cite{Zufelt14}, show that the bridge number must
be at least $6$.

It should be noted that in \cite{Hoffman98}, Hoffman mentions having
unpublished notes proving the Cabling Conjecture for $5$-bridge knots.

%% file: sections/Setup.tex
\section{Setup}

Let $k\subset S^3$ be a knot, and let $M = S^3 - N(k)$ be the exterior
of $k$. Let $\gamma$ be the meridional slope in $\bdd M$. Given
a slope $\pi$ in $\bdd M$, let $M(\pi )$ be the $3$-manifold
obtained by performing $\pi$-Dehn surgery on $k$.

Let $P$ be a $2$-sphere in $M(\pi)$ which intersects $k'$ transversely,
and let $\check{P} = P \cap M$. Then each component of $\bdd \check{P}$
has slope $\pi$. Let $\gamma$ be the meridian slope of $\overline{N(k)}$.
Then $\gamma$-Dehn surgery is trivial, so $M(\gamma) \cong S^2$.
If $Q$ is a $2$-sphere in $M(\gamma)$ intersecting $k$ transversely,
then each component of $\bdd \check{Q}$ has slope $\gamma$.

We use Gabai thin position from \cite{Gabai87} (this setup follows \cite{Hoffman95}).
Define a height function for $k\subset S^3$,
$h:S^3 \setminus \{x,y\}\cong S^2 \times \mathbb{R}\to \mathbb{R}$.
Let $Q_\alpha = h^{-1}(\alpha) = S^2\times \alpha$ be the
\define{level sphere} at height $\alpha$.
A \define{generic presentation} of $k$ is an embedding
$f:S^1\to S^3$ such that $h\circ f$ is a Morse function with $2\beta$
critical points occuring at distinct levels.
Assuming a generic presentation of $k$,
let $c_1 < c_2 < \ldots < c_{2\beta}$ be the critical levels and select
levels $h_1 < h_2 < \ldots < h_{2\beta}$ such that $c_i < h_i < c_{i + 1}$
for $i = 1,\ldots,2\beta - 1$.
Define $Q_i$ to be the level sphere at level $h_i$.
The \define{width} is
$$w(k) = min\{\frac{1}{2}\sum\vert Q_i \cap f(S^1) \vert
\vert f \text{ is a generic presentation of }k\}.$$
A \define{thin presentation} is one which realizes $w(k)$ (i.e. a generic
presentation which minimizes the sum).

Assume a thin presentation of $k$.
The following is an combination of results of Gabai \cite{Gabai87} with
results of Gordon and Luecke \cite{GordonLuecke87}, as stated in
\cite{Hoffman95}.
\begin{prop}[Proposition 1.2.1 in \cite{Hoffman95}]
\label{prop:GabaiGL}
Suppose $P\subset M(\pi)$ is an essential $2$-sphere (with
$\check{P} = P\cap M$) such that $P$ meets $k'$ transversely and minimally.
Then there is a (level) $2$-sphere $Q\subset M(\gamma)$ (with
$\check{Q} = Q\cap M$) such that
\begin{enumerate}[(i)]
\item $\bdd \check{P} \subset \bdd M$ (resp.
$\bdd \check{Q} \subset \bdd M$)
consists of parallel copies of $\pi$ (resp. $\gamma$);
\item $\check{P}$ and $\check{Q}$ intersect transversely;
\item no arc of $\check{P} \cap \check{Q}$ is boundary-parallel in either
$\check{P}$ or $\check{Q}$; and
\item \label{itm:intNumberOne} each component $\bdd \check{P}$
meets each component of $\bdd \check{Q}$ exactly once.
\end{enumerate}
\end{prop}

This relates to the Cabling conjecture as shown below.
Let $p = \vert P \cap k' \vert$ and $q = \vert Q \cap k \vert$.
\begin{prop}[Proposition 1.2.2 from \cite{Hoffman95}]
 If $P$ is a $2$-sphere in $M(\pi)$ such that $P$ meets $k'$ transversely
 and $p = 2$, then the knot $k$ is a cabled knot and $\check{P}$ is the
 cabling annulus.
\end{prop}

Our goal is to prove the Cabling Conjecture for bridge number $b \leq 5$.
We would like to claim that because $k$ is in thin position, $q \leq 2b$.
To see this, note that with a $b$-bridge knot $k$ in (generic) bridge
position $\mathcal{B}$,
$$\sum\vert Q_i \cap f(S^1) \vert =
2 + 4 + \ldots + (2b - 2) + 2b + (2b - 2) + \ldots + 4 + 2.$$
If $k$ is put in a different generic position $\mathcal{P}$ such that
some level sphere intersects $k$ in
$q' > 2b$ points, we can compare the widths of the two presentations
\begin{align}
w_\mathcal{P}(k) &= 2 + 4 + \ldots + (q' - 2) + q' + (q' - 2) + \ldots + 4 + 2 + \ldots\\
w_\mathcal{B}(k) &= 2 + 4 + \ldots + (2b - 2) + 2b + (2b - 2) + \ldots + 4 + 2.
\end{align}
Clearly $w_\mathcal{B}(k) < w_\mathcal{P}(k)$, 
so $\mathcal{P}$ is not a thin presentation of $k$.

Thus we assume $p > 2$ and want to show that $q > 10$.

%% file: sections/IntersectionGraphs.tex
\section{Intersection Graphs}

We will henceforth assume that $p > 2$.

\subsection{Basic Definitions}

We now define the graphs $G_P \subset P$ and $G_Q\subset Q$.
Many definitions and theorems will apply to both spheres $P$ and $Q$.
We will use $\{\redS,\lvlS\} = \{P,Q\}$ when we wish to make statements
which may apply to either sphere.
A \define{fat vertex} of $\redG$ is a component of
$\bdd \check{\redS}$, and each arc component of
$\check{P} \cap \check{Q}$ is an edge in $\redG$.
Select a fat vertex in $\redG$ to label
$1$, and follow along the appropriate knot ($k'$ if $\redS = P$, $k$ if
$\redS = Q$), labeling the remaining fat vertices in the order in which they
are encountered.
We will denote by $V_\redS$ the vertex set of $\redG$.

Since edges of $\redG$ are arc components of $\check{P} \cap \check{Q}$,
they meet fat vertices of $\lvlG$ at components of
$\bdd \check{\lvlS}$. This can
be used to \define{label} the points at which edges meet fat vertices in
$\redG$ with the labels of the appropriate components of
$\bdd \check{\lvlS}$.
We will frequently refer to subsets of $V_P$ by variations of $V$, and subsets
of the labels in $G_P$ by variations of $L$. Since vertices in $\redG$ are
labels in $\lvlG$, this means we will frequently refer to subsets of
$V_Q$ by variations of $L$ (being labels in $G_P$), and subsets of the
labels of $G_Q$ by variations of $V$ (being labels in $G_P$).

Note that by Proposition \ref{prop:GabaiGL}(\ref{itm:intNumberOne}),
each label from $\bdd \check{\lvlS}$ appears
precisely once on each fat vertex in $\redG$. Furthermore, since components
of $\bdd \check{\lvlS}$ are curves of the same slope on $\bdd M$,
the labels from $\bdd \check{\lvlS}$ will appear in the same order
around every fat vertex of $\redG$, though in either orientation
(clockwise or counter-clockwise).
Fat vertices may be assigned a \define{sign} depending on the orientation
of the labels from $\bdd \check{\lvlS}$ (see Figure \ref{fig:GPFatVertices}).
Vertices in $\redG$ are \define{parallel} if they have the same sign,
and \define{antiparallel} if their signs differ. Let $V$ ($V'$)
be a set of vertices of $\redG$ such that all vertices of $V$ ($V'$) are
parallel. Then we call $V$ parallel to $V'$ if the sign which the
vertices in $V$ share matches the sign of the vertices in $V'$, and we
call $V$ antiparallel to $V'$ if the signs are opposite.
\begin{figure}[h]
 \centering
 \includegraphics[width=0.8\textwidth]{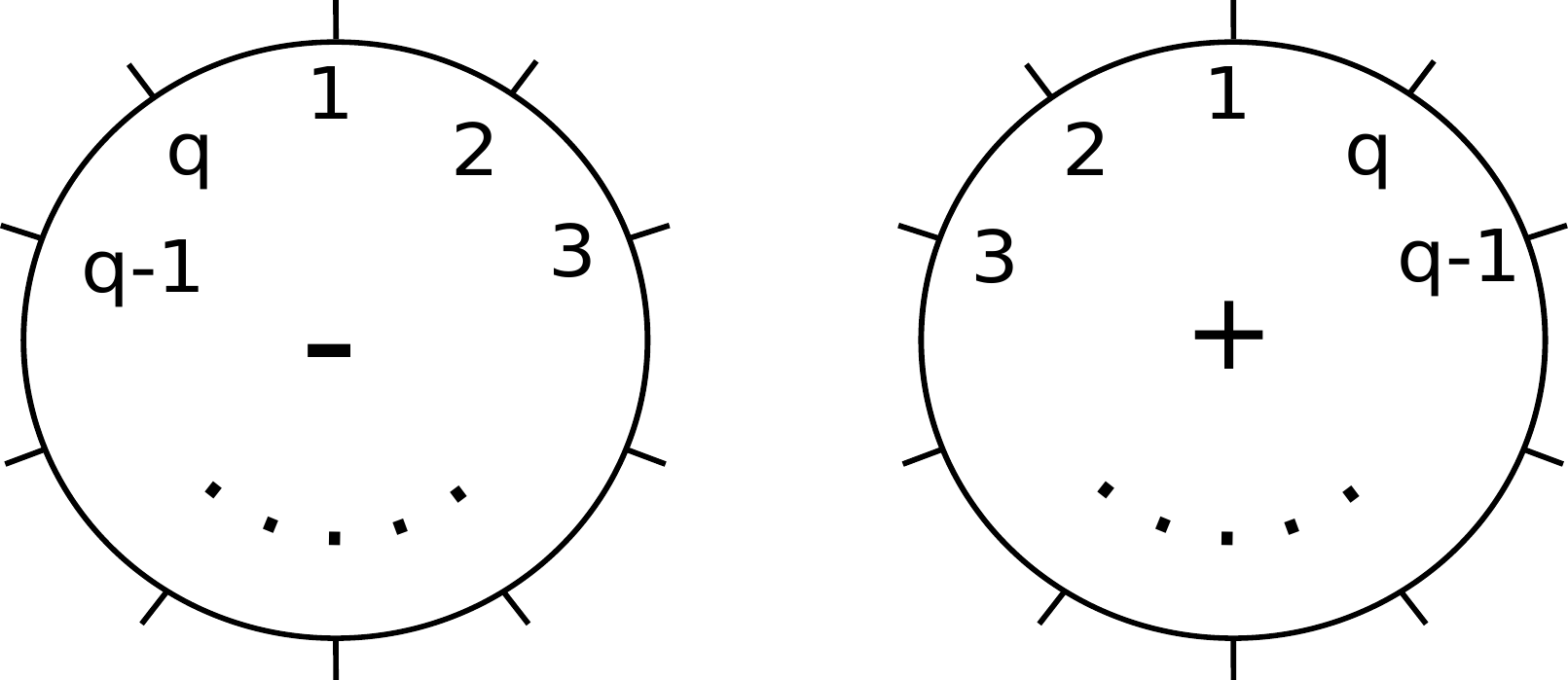}
 \caption{Fat vertices in $\redG$, of each sign.}
 \label{fig:GPFatVertices}
\end{figure}

To each label $x$ on a fat vertex we associate a \define{parity}, which is
the sign of $x$ as a vertex in the other graph. Every label and vertex
pair $(x,v)$ may therefore be assigned a \define{character}, where
char$(x,v) = ($parity $x)($sign $v)$.

It follows from orientability of $M$, $P$, and $Q$ that edges in either
graph connect $(x,v)$ pairs of opposite character. This is known as the
\define{parity rule}.

We will frequently refer to various subgraphs of $\redG$. If $E$
is a collection of edges, $\redG(E)$ is the subgraph of $\redG$ containing all
the fat vertices of $\redG$, and edges $E$.
Let $V,W\subset V_\redS$ be sets of vertices of $\redG$. Then $[V,W]$ is the 
set of edges in $\redG$ between a $V$ vertex and a $W$ vertex, or equivilantly
the set of edges in $\lvlG$ between a $V$ label and a $W$ label.
Let $L$ be a set of labels of $\redG$. We will often consider the subgraph
$\redG([L, V_\lvlS])$, which we will abbreviate as $\redG(L)$.
Note that for an arbitrary label set $L$, $\redG(L)$ may (in fact, generally
will) have edges meeting vertices at labels outside $L$. Such labels are
\define{exceptional labels}.

 \begin{figure}[!h]
  \centering
  \includegraphics[width=0.95\textwidth]{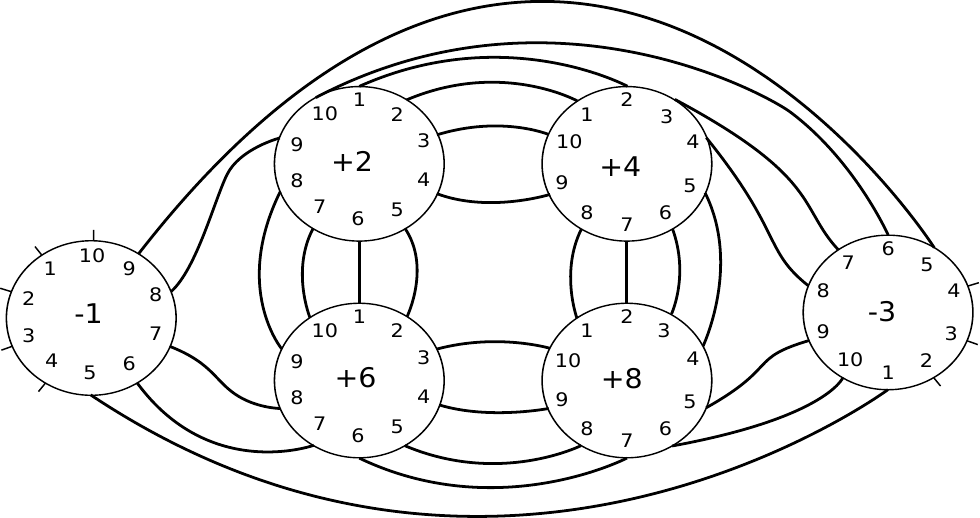}
  \caption[An example of a great web in $G_Q$.]{An example of part
  of $G_Q$. There is a Scharlemann cycle on
  the vertices $\{2,4\}$ and a new $5$-cycle $\sigma$ on the vertices
  $\{1,3\}$. The $\sigma$-disk depicted as bounded contains the
  $\sigma$-set $\{2,4,6,8\}$, which also happens to be an
  innermost $(+)$-set and a great web.}
  \label{fig:GreatWeb}
 \end{figure}
An \define{$x$-cycle} in $\redS$ is a directed cycle on parallel vertices
such that the tail of each edge in the cycle is at the label $x$ on each
vertex. A \define{great $x$-cycle} is an $x$-cycle which bounds a disk such
that all vertices inside the disk are parallel to the $x$-cycle.
A \define{Scharlemann cycle} is a great $x$-cycle bounding a disk containing
in its interior no vertices or edges. A \define{corner} of a face $F$ is the
part of the boundary of a fat vertex which is incident to $F$. If $x$ and
$y$ are labels on a vertex $v$ of $\redS$, the two complementary pieces of
the boundary of $v$ extending from the label $x$ to the label $y$ are
both $\langle x,y \rangle$ corners.
Since a Scharlemann cycle bounds a disk containing no vertices or edges
of $\redG$, each
corner in a Scharlemann cycle must be an $\langle x,x+1\rangle$ corner.
We may therefore refer to an $x,x+1$-Scharlemann cycle.
The \define{order} of a Scharlemann cycle is the number of vertices the
cycle traverses.
A \define{new $x$-cycle} is an $x$-cycle which is not a Scharlemann cycle.
 \begin{lem}[Lemma 2.1.2 of \cite{Hoffman95}]
 \label{lem:ScharlemannOrder}
 If, in $G_Q$, $\Sigma_1$ is an $x_1,x_2$-Scharlemann cycle of order $m$
 and $\Sigma_2$ is a $y_1,y_2$-Scharlemann cycle of order $n$, then
 $\langle x_1,x_2 \rangle = \langle y_1,y_2 \rangle$ and $m=n$.
\end{lem}

The key technical result in the resolution of the Knot Complement Problem
was the existence of Scharlemann cycles. Gordon and Luecke's proof
is described in a later section, as many lemmas in their proof are
integral to our result, which explores the graph structure indicated
by their inductive proof.
 \begin{thm}[\cite{GordonLuecke89}]
 \label{thm:SchCycleExistence}
  $G_Q$ contains a Scharlemann cycle.
 \end{thm}
 By Theorem \ref{thm:SchCycleExistence}, $G_Q$ must have a Scharlemann cycle. By
convention the labels of $G_Q$ and therefore the vertices of $G_P$
are labeled such that Scharlemann cycles of $G_Q$ are $1,2$-Scharlemann
cycles, since by Lemma \ref{lem:ScharlemannOrder} all Scharlemann cycles
are on the same labels. The vertices $1$ and
$2$ in $G_P$ are the \define{special vertices}.
The remaining vertices $V_r$ of $G_P$ are the \define{regular vertices}.

 Hoffman's important result, meanwhile, was the nonexistence of new great
 $x$-cycles.
 \begin{thm}[Lemma 3.0.3 of \cite{Hoffman95}]
 \label{thm:NoNewGreatXCycle}
 If $p > 2$, then $G_Q$ does not contain a new great $x$-cycle.
\end{thm}

A vertex $x$ of $G_Q$ is \define{isolated} if each edge incident to $x$ leads
to an antiparallel vertex.
\begin{lem}[2.4.1 from \cite{Hoffman95}]
\label{lem:NoIsolatedVertices}
 There are no isolated vertices in $G_Q$.
\end{lem}

 A \define{great $m$-web} $\Lambda$ is a collection of parallel vertices
in $\redG$ such that
\begin{enumerate}
\item $\Lambda$ lies in a disk $D_\lambda$ of $\redS$ such that every vertex
in $D_\Lambda$ is a vertex of $\Lambda$, and
\item precisely $m$ edges leave $\Lambda$.
\end{enumerate}
A great $(p - 2)$-web in $G_Q$ will be referred to simply as a
\define{great web}.
The proof of Theorem \ref{thm:SchCycleExistence} in \cite{GordonLuecke89}
shows the existence of a great web in $G_Q$.

Only the relevent parts of the following results have been kept.
\begin{thm}[Part of Proposition 3.1 from \cite{Zufelt14}]
 \label{thm:GreatWebDivisibility}
 Let $\Lambda$ be a great web, $V(\Lambda)$ the vertices of the great web,
and $v = \vert V(\Lambda)\vert $. Let $n$ be the Scharlemann order.
Then $n$ divides $v$.
\end{thm}

\begin{thm}[Corollary 3.2 from \cite{Zufelt14}]
 \label{thm:SchCycleNotWeb}
 Under the hypotheses from Theorem \ref{thm:GreatWebDivisibility}, $n\neq v$.
 Hence the Cabling Conjecture holds for knots with bridge number $\leq 3$,
 and modulo the case $n = 2$, $v = 4$, for knots with $b\leq 5$.
\end{thm}

\begin{cor}
 \label{cor:MinWebSize}
 If $\Lambda$ is a great web, $\vert V(\Lambda) \vert \geq 4$.
\end{cor}
\begin{proof}
 The Scharlemann order $n$ must be at least $2$, and
 $\vert V(\Lambda)\vert > n$ must be a multiple of $n$.
\end{proof}

Our method is to establish minimums for $q$ in different cases by
finding several great webs, and applying Corollary \ref{cor:MinWebSize}.

Let $V$ be a set of vertices of $\redG$.
Suppose $\sigma$ is a cycle on parallel vertices in $\redG$ bounding a disk $D$
which contains the vertices $V$. Then $D$ is a \define{$\sigma$-disk} and $V$
is a \define{$\sigma$-set}. We call $D$ \define{nontrivial} if
$V$ contains elements of the opposite sign of $\sigma$, and $\sigma$
nontrivial if it bounds no trivial $\sigma$-disk.
 Let $s\in \{+,-\}$. If every vertex
 in $V$ has sign $s$ and every edge leaving $V$ goes to a vertex of sign $-s$,
 $V$ is an \define{$(s)$-set} of vertices. An \define{$(s)$-disk}
 is a disk $D\subset \redS$ containing a nonempty connected $(s)$-set $V$
 and no other vertices, such that all edges $[V,V]$ are also contained in $D$.
 An $(s)$-set contained in an $(s)$-disk is called an \define{innermost
 $s$-set}. $V^*$ will denote the set of labels at which edges leave $V$. 

Let $V$ be a subset of vertices in $\redG$, and $L$ a subset of labels.
$\redG$ has the parallel property $\mathbb{P}(V,L)$
if for each vertex $x$ of $V$ there exists a label $y(x)$ of $L$ such that
the edge of $\redG$ incident to $x$ at $y(x)$ goes to a vertex parallel to $x$.
$\redG$ has the antiparallel property $\mathbb{A}(V,L)$ if for each label $y$
of $L$ there exists a vertex $x(y)$ of $V$ such that the edge of $G$
incident to $x$ at the label $y(x)$ connects $x$ to an antiparallel vertex.
Note that if $V_1$ is a vertex subset of $\redG$ and $V_2$ is a vertex subset
of $\lvlG$, then $\redG$ has $\mathbb{A}(V_1,V_2)$ if and only if $\lvlG$ has
$\mathbb{P}(V_2,V_1)$.

By basic graph theory,
\begin{prop}
 \label{prop:AsManyEdgesAsVertsImpliesCircuit}
 If $G$ contains no cycles, $G$ contains more vertices than edges.
\end{prop}

\begin{prop}
 \label{prop:EnoughEdgesImpliesCircuit}
 Let $G$ be a directed graph, let $V$ be a subset of the vertices of $G$.
 If for each $v\in V$ there is a unique edge $e_v$ between vertices of $V$,
 such that the tail of $e_v$ is at $v$, then $G$ has an directed cycle on the
 vertices $V$.
\end{prop}
\begin{proof}
 Since $\vert \{e_v\} \vert = \vert V \vert$, by Proposition
 \ref{prop:AsManyEdgesAsVertsImpliesCircuit}, $G$ contains a cycle. The
 orientations of the edges in the cycle must agree by uniqueness of each
 $e_v$.
\end{proof}

\begin{cor}
\label{cor:NoXEdgesLeavingImpliesXCycle}
 Let $V\subset V_Q$ be a set of uniform sign. If there is a label $x_0\in V_P$ such that no
 edge leaves $V$ at a label $x_0$, then $V$ contains an $x_0$-cycle.
\end{cor}
\begin{proof}
 Apply Proposition \ref{prop:EnoughEdgesImpliesCircuit} to the components of
 $G_Q(\{x_0\})$ on the vertices in $V$.
\end{proof}
 
 \begin{lem}
 \label{lem:SDiskProperties}
  Let $V$ be an innermost $(s)$-set in $G_Q$. Then all of the following
  hold:
  \begin{enumerate}
   \item $\vert V \vert \geq 2$,
   \item $V^* \supset V_r$,
   \item $G_P$ has $\mathbb{P}(V_r, V)$.
  \end{enumerate}
  Furthermore, if $V$ does not have a Scharlemann cycle, then $V^* = V_P$.
 \end{lem}

\begin{proof}
 Suppose $x_0\in V_r$ is not in $V^*$. Then there is a new $x_0$-cycle $\Sigma$ 
 on $V$, by Corollary \ref{cor:NoXEdgesLeavingImpliesXCycle}.
 By definition of $(s)$-disk, $\Sigma$ is a great new $x_0$-cycle,
 a contradiction with Theorem \ref{thm:NoNewGreatXCycle}.
 Therefore no such $x_0$ exists.
 If $V = \{v\}$, then since $v$ cannot be isolated
 (Lemma \ref{lem:NoIsolatedVertices}),
 there is a loop on $v$, which must also be a new great $x$-cycle.
 Thus $\vert V \vert \geq 2$.
 Since $V^* \supset V_r$, $G_Q$ has $\mathbb{A}(V, V_r)$, so $G_P$ has
 $\mathbb{P}(V_r,V)$.
 
 Suppose $V$ does not have a Scharlemann cycle. Then $V$ cannot have any
 $x$-cycles, so $V^* = V_P$, by Corollary
 \ref{cor:NoXEdgesLeavingImpliesXCycle}.
\end{proof}

\begin{lem}
 \label{lem:EdgesInSDisks}
 Let $V$ be an innermost $(s)$-set. If $\vert [V,V_Q\setminus V] \vert = p - 2$,
 then for each $x\in V_r$, there are precisely $\vert V \vert - 1$ edges from
 $[V, V]$ adjacent to $x$.
\end{lem}
\begin{proof}
 By Lemma \ref{lem:SDiskProperties}, $V^* \supset V_r$. So if
 $\vert [V,V_Q\setminus V] \vert = p - 2$, precisely one edge leaves $V$ from
 each label $x\in V_r$. Fix $x_0\in V_r$. Each of the $\vert V\vert$ vertices in
 $V$ have an $x_0$ label, so $\vert V \vert - 1$ of the edges adjacent to $x_0$
 labels must be in $[V,V]$.
\end{proof}

Note that a consequence of Corollary \ref{cor:NoXEdgesLeavingImpliesXCycle}
and Theorem \ref{thm:NoNewGreatXCycle}
is that if $\Lambda$ is a great web, there is precisely one edge leaving
$\Lambda$ from each label in $V_r$, and $\Lambda$ must contain a 
Scharlemann cycle. Lemma \ref{lem:SDiskProperties}, meanwhile, shows
that for an innermost $(s)$-set $V$, there is at least one edge leaving
$V$ from each label in $V_r$. The two kinds of sets are similar, but an
innermost $(s)$-set $V$ has the property that all edges leaving $V$ go to
antiparallel vertices but there are an unknown number of such edges,
while a great web $\Lambda$ has the property that those number and
configuration of edges leaving $\Lambda$ are well known, but those
edges may go to parallel vertices.

\begin{lem}
 \label{lem:AntiparallelEdgesInQGiveCircuitInP}
 Let $V$ be any set of vertices in $G_Q$. Let $E_{antipar}$ be the set of edges
 joining vertices of $V$ to antiparallel vertices (possibly also in $V$).
 If $\vert E_{antipar} \vert > p - 2$, then $G_P(V)$ has a cycle on parallel
 vertices.
 
 Furthermore, if for each $x\in V_r$, there is a $v\in V$ and $e\in E_{antipar}$
 such that $e$ meets $v$ at label $x$, then the cycle can be oriented with
 the tail of each edge meeting the vertex in $G_P(V)$ at a label in $V$.
\end{lem}

\begin{proof}
 If $\vert E_{antipar} \vert > p - 2$, then $G_P(V)$ contains more than
 $p-2$ edges between parallel vertices. Without loss of generality we may
 therefore assume that $G_P(V)$ contains at least $\frac{p}{2}$ edges between
 the $\frac{p}{2}$ positive vertices.
 Let $G$ be the subgraph of $G_P(V)$ obtained by
 discarding edges between antiparallel vertices. Some positive component of $G$
 has at least as many edges as vertices, and by Proposition
 \ref{prop:AsManyEdgesAsVertsImpliesCircuit}, the component will have a cycle.
\end{proof}

Let $X,Y\subset V_Q\times V_P$, $r\in V_P$. A \define{$(X \to Y \to r)$ tree}
is a subtree of $G_P$ rooted at $r$ such that each edge is directed from
a (vertex, label) pair in $X$ to a (vertex, label) pair in $Y$.
If $V,W\subset V_Q$, a $(V\to W\to r)$ tree means a
$(V\times V_P\to W\times V_P\to r)$ tree.
If $X\subset V_Q\times V_P$, a $(X\to r)$ tree means a 
$(X\to V_Q\times V_P\to r)$ tree.

\begin{lem}
 \label{lem:APropertyImpliesCircuitsOrTrees}
 Let $V$ be an innermost $(s)$-set in $G_Q$.
 Let $W$ be the vertices in $V_Q \setminus V$ which have edges to $V$.
 Then any component of $G_P([V,W])$ that is not in a $(V\to W \to r)$ tree
 on parallel vertices with $r$ one of the special vertices, must have a
 cycle on parallel vertices.
\end{lem}
\begin{proof}
 By Lemma \ref{lem:SDiskProperties}, $V^* \supset V_r$, which means that
 $\vert [V, W] \vert \geq p - 2$. Since $V^* \supset V_r$, for each $x\in V_r$
 there exists an edge $e_x$ leaving $x$ at a $V$ label and going to a parallel
 vertex at a $W$ label.
 Suppose that there does not exist
 a path from some negative $v_0\in V_r$ to $1$ along edges with
 tails at labels in $V$ and heads at labels in $W$. Since every
 $v\in V_r$ has such an edge, the absence of such a path
 implies that there is an oriented cycle of such edges.
\end{proof}

\begin{lem}
\label{lem:SigmaDisksHaveSDisks}
 Nontrivial $\sigma$-disks contain $(s)$-disks.
\end{lem}
\begin{proof}
 Let $D_i$ be a nontrivial $\sigma_i$-disk containing the $\sigma_i$-set
 $L_i$, and let $V_i$ be an $(s_i)$-set contained in $L_i$. Note that
 such an $(s_i)$-set must exist, or $D_i$ is trivial. If $V_i$ is not
 contained in an $(s_i)$-disk, then there is a cycle $\sigma_{i + 1}$ on
 $V_i$ that bounds a nontrivial $\sigma_{i + 1}$-disk $D_{i + 1}\subset D_i$.
 The $\sigma_{i + 1}$-set $L_{i + 1}$ contains an $(s_{i + 1})$-set
 $V_{i + 1}$.
 
 Since $\sigma_i \cap \sigma_{i + 1} = \emptyset$, this can only be
 repeated finitely many times, until an $(s_n)$-set $V_n$ is reached that is
 contained in an $(s_n)$-disk.
\end{proof}

\begin{cor}
\label{cor:NewXCyclesHaveSDisks}
 If $\Sigma$ is a new $x$-cycle, any $\Sigma$-disk contains an $(s)$-disk.
\end{cor}

\subsection{Dual Orientation}

 \begin{figure}[h]
  \centering
  \includegraphics[width=0.8\textwidth]{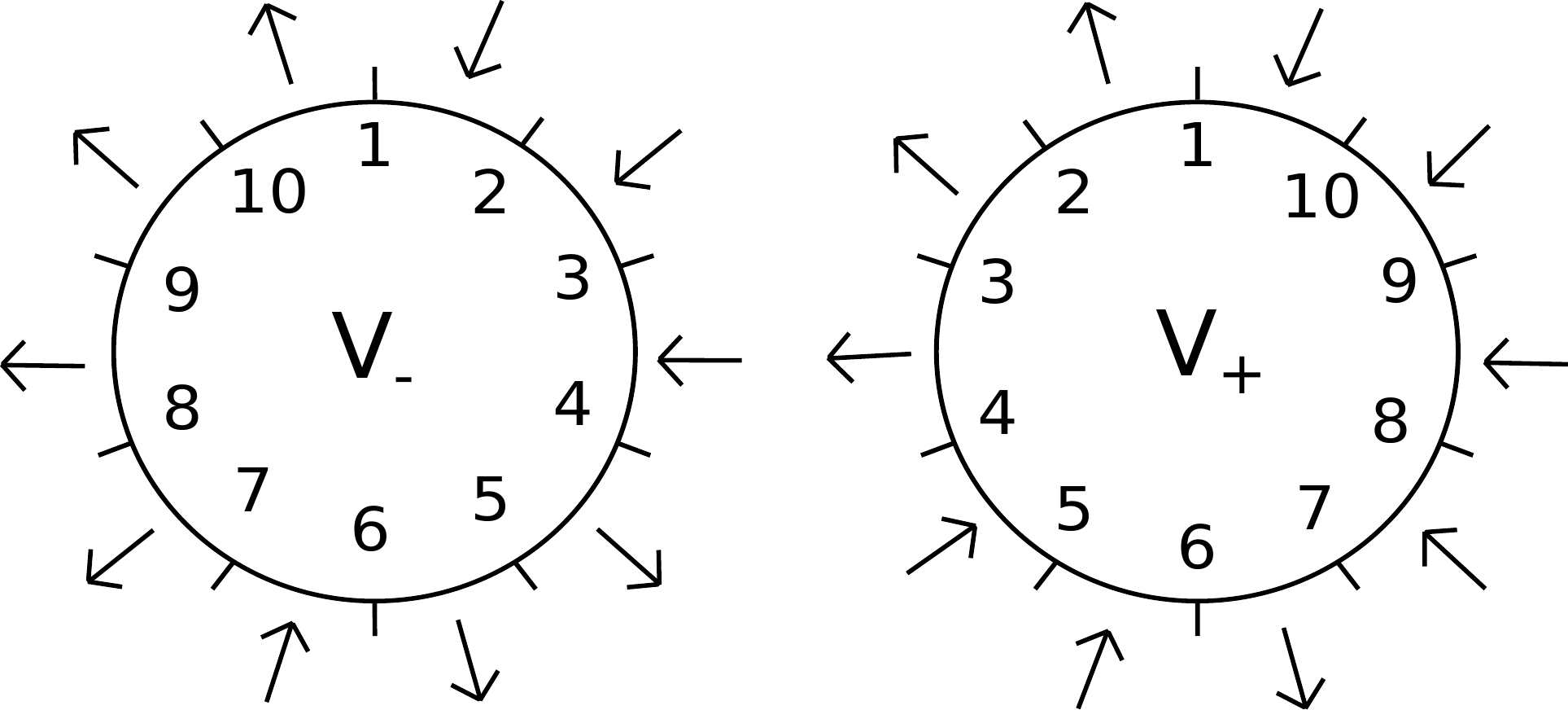}
  \caption[Model Vertices.]{Model vertices. Note that $A = \{4,7\}$, $C = \{1,6\}$ on both vertices.}
  \label{fig:ModelVertices}
 \end{figure}
A pair of \define{model fat vertices} are abstract fat vertices $V_\pm$ with
marks on the boundary representing edges incident to each vertex, labels at
each mark, and a sign $\pm$ indicated by the subscript.
$V_-$ is the reflection of $V_+$.
If $L$ is a subset of labels on an abstract fat vertex $V$, an $L$-interval
is an $\langle x,y \rangle$ corner where no labels of $L$ are in the
interior of $\langle x,y \rangle$.
A \define{dual orientation} on an $L$-interval is indicated by an arrow
pointing either into (a sink) or out from (a source) the corner.
If all dual orientations on a vertex are the same, the vertex is said to
have \define{uniform dual orientation}, or be a sink or source (depending
on the dual orientation).
A \define{star} $T$ is an ordered triple $(V(T), L(T), \omega (T))$,
where $V(T) = V_\pm$, $L(T)$ is a subset of the labels around
$V(T)$, and $\omega(T)$ is an assignment of dual orientations to each
$L(T)$-interval around $V(T)$. Given a star $T$, $\overline{T}$ is the same
star with reversed dual orientations, and $-T$ is the mirror image of
$\overline{T}$.

 Let $\Delta$ be a nonempty face of $G_P$ (meaning its interior contains
 vertices). The set of labels on the corner of vertex $i$ in $\bdd \Delta$
  will be denoted $X_i(\Delta)$. Furthermore $X_-(\Delta) =
 \bigcup X_i(\Delta)$ where the union is taken over negative vertices in
 $\bdd \Delta$. $X_+(\Delta)$ is defined similarly.

  \begin{figure}[h]
  \centering
  \includegraphics[width=0.8\textwidth]{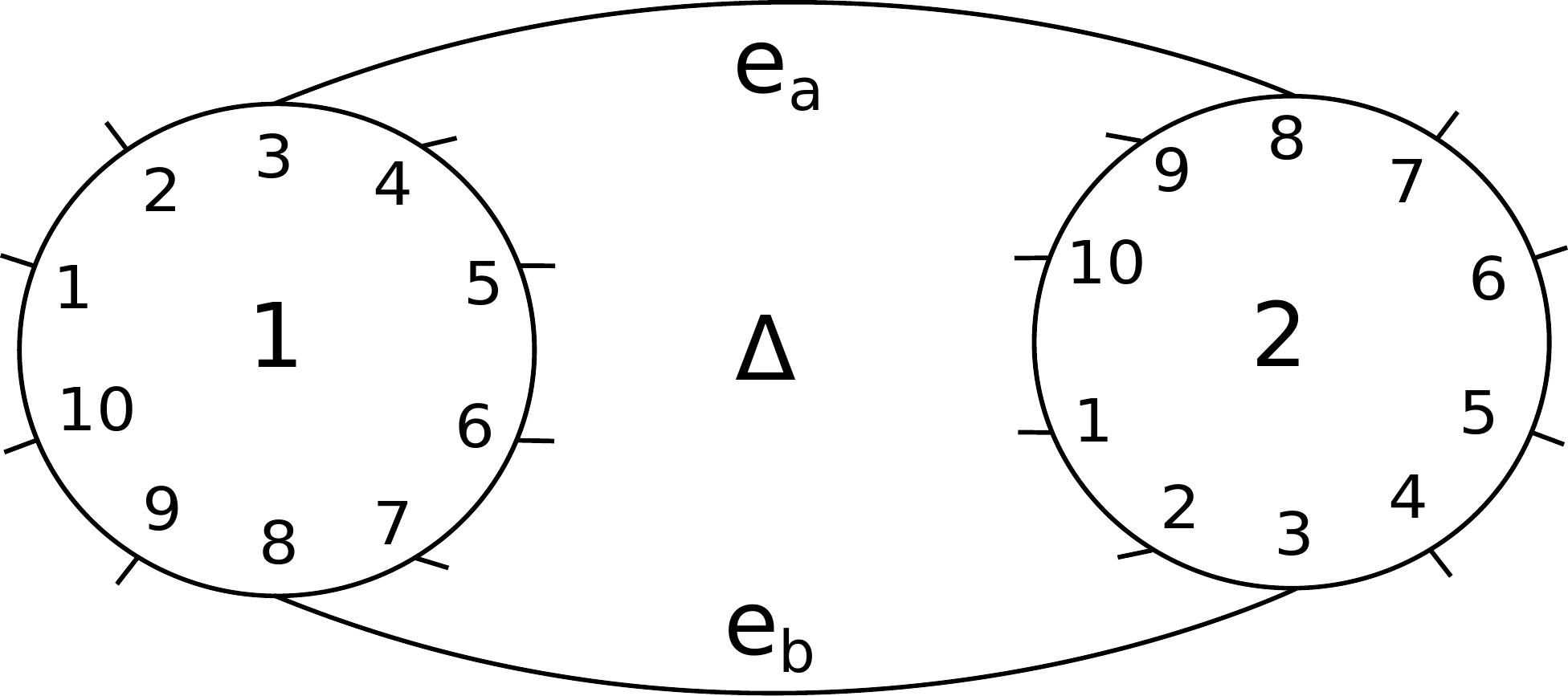}
  \caption[Example of part of $G_P$ for Lemma
          \ref{lem:DeltaCornersHaveDisjointLabels}.]{
           Example of part of $G_P$ for Lemma
           \ref{lem:DeltaCornersHaveDisjointLabels}.
           The edges $e_a$ and $e_b$ come from a Scharlemann cycle of
           order $2$ in $G_Q$ on the vertices $3$ and $8$.}
  \label{fig:DisjointSchDisks}
 \end{figure}
\begin{lem}
 \label{lem:DeltaCornersHaveDisjointLabels}
 Let $E$ be the edges of a Scharlemann cycle in $G_Q$.
 Let $\Delta$ be a disk face of $G_P(E)$ with no edges of $G_P(E)$ in its
 interior. 
 Then considered relative to $G_P$,
 $X_-(\Delta) \cap X_+(\Delta) = \emptyset$.
\end{lem}
\begin{proof}
 Note first that $\Delta$ is a $2$-corner face with corners on the
 special vertices.
 Let the edges in $\bdd \Delta$ be $e_a$ and $e_b$. The edge $e_a$ meets
 $1$ and $2$ at labels $l_a^1$ and $l_a^2$ respectively, and the edge
 $e_b$ meets $1$ and $2$ at labels $l_b^1$ and $l_b^2$.
 Suppose that $x\in X_-(\Delta) \cap X_+(\Delta)$. Then the labels
 corresponding to $G_Q$ vertices in the Scharlemann cycle which are
 on either side of $x$ on the star $T$, are $\{s_1,s_2\}$.
 But then $\{s_1,s_2\} = \{l_a^1,l_b^1\} = \{l_a^2,l_b^2\}$.
 Since $1$ is a negative vertex and $2$ is positive, the order of the 
 labels is reversed, thus for $x$ to be in $X_-(\Delta) \cap X_+(\Delta)$,
 we must have $l_a^1 = l_a^2$ and $l_b^1 = l_b^2$, implying that
 $e_a$ and $e_b$ are two loops in $G_Q$ rather than a
 Scharlemann cycle.
\end{proof}

 \begin{figure}[h]
  \centering
  \includegraphics[width=0.8\textwidth]{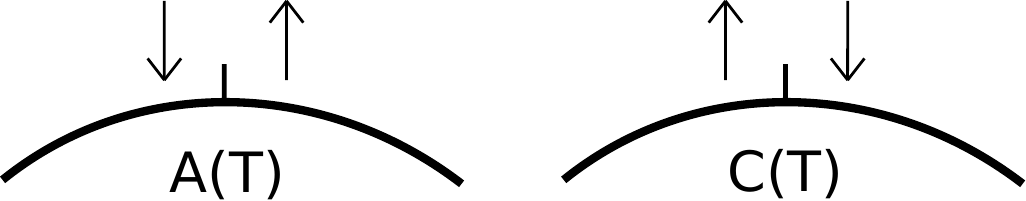}
  \caption{Switch labels.}
  \label{fig:SwitchLabels}
 \end{figure}
We define a label of a fat vertex to be an \define{anticlockwise switch} if
the dual orientation in the clockwise direction from the label is out while
the dual orientation in the counter-clockwise direction from the label is in.
We define a label of a fat vertex to be a \define{clockwise switch} if the
dual orientations are opposites those of an anticlockwise switch.
For a star $T$,
$A(T) = \{l\in L(T) \vert l \text{ is an anticlockwise switch}\}$.
Likewise, $C(T) = \{l\in L(T) \vert l \text{ is a clockwise switch}\}$.
Let $S(T) = A(T)\cup C(T)$.
For $l\in L(T) \setminus S(T)$, let $\phi(l) = +$ if both
adjacent dual orientations are out, and $\phi(l) = -$ if both adjacent
dual orienations are in. Define
$$B_s(T) = \{l\in L(T) \setminus S(T) \vert
            (sign\,T)(parity\,l)(\phi(l)) = s \}$$
for $s\in \{+, -\}$. The labels on $T$ are partitioned into
$A(T)$, $C(T)$, $B_+(T)$, and $B_-(T)$. A star $T$ is \define{coherent} if
all elements of $A(T)$ have the same parity and all elements of
$C(T)$ have the same parity. Note that $A(T) = A(-T)$ and $C(T) = C(-T)$.

An $m$-type is a tuple of signs ($\{+,-\}$) of length $m$. A
\define{trivial} $m$-type is one of uniform sign. Let $L$ be a set of
labels, and $l$ the number of $L$-intervals on a star $T$.
An $L$-type is an $l$-type where each sign corresponds to a
unique $L$-interval.
Let $\mathbb{L}_0$ be a subset of the $L$-intervals, and let $\tau$ be an
$L$-type. Then $\tau \vert \mathbb{L}_0$ is the
$\vert \mathbb{L}_0 \vert$-type obtained by restricting $\tau$ to the
$L$-intervals in $\mathbb{L}_0$.
Let $T$ be a star and $\tau$ be an $L(T)$-type. If it is possible to assign
distinct signs to the dual orientations such that $\omega(T)$ and $\tau$
match (e.g. sink equals $-$ and source equals $+$, or switched),
then $T$ \define{represents} $\tau$, and we write $[T] = \tau$.
We will say that an $L$-type $\tau$ is coherent if there exists a coherent
star such that $[T] = \tau$.

Let $D$ be a disk face of $G_P(L)$ and let $\tau$ be an $L$-type. 
$D$ \define{represents} $\tau$ if there is a star $T$ with $[T] = \tau$
such that with the dual orientation of $G_P(L)$ inherited from $T$,
$D$ is a sink or source (i.e. all corners on $\bdd D$ have matching
dual orientations). $G_P(L)$ represents $\tau$ if it has a disk face
which represents $\tau$.

The \define{positive} and \define{negative (clockwise) derivatives of $T$}
are
$$d^\pm (T) = (V(T), C(T), \omega(d^\pm))$$
where $\omega (d^\pm)$ on a $C(T)$-interval $I$ is defined as follows.
Let $a\in A(T)$ be the unique anticlockwise switch in $I$. Then
$\omega(d^+ T)$ is a source if $char\, a = +$, and a sink if $char\, a = -$.
On the other hand, $\omega(d^- T)$ is a sink if $char\, a = +$, and a
source if $char\, a = -$. By $d$ we mean either $d^+$ or $d^-$.
 \begin{figure}[h]
  \centering
  \includegraphics[width=0.8\textwidth]{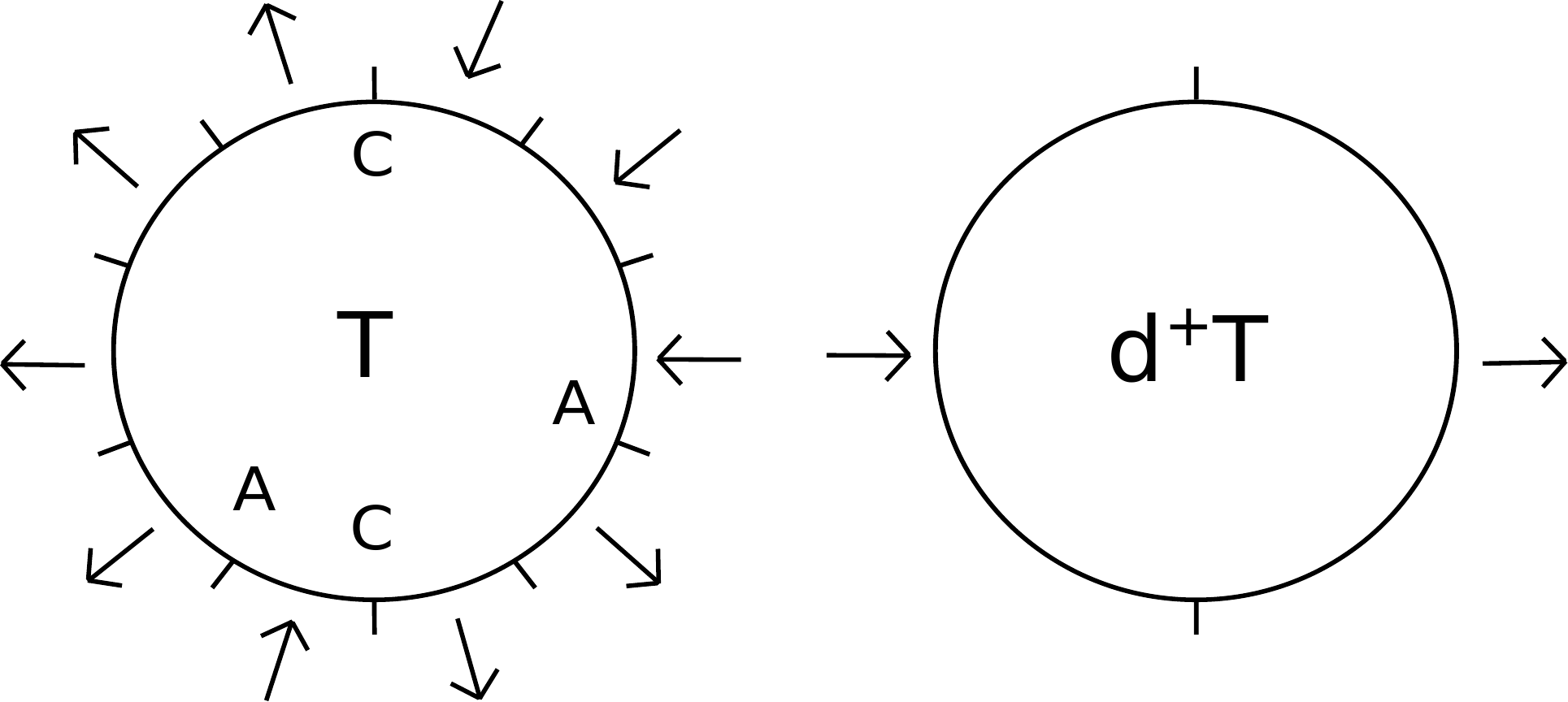}
  \caption{The positive (clockwise) derivative of a star.}
  \label{fig:Derivative}
 \end{figure}

Let $L_0$ be a subset of the labels on $V(T)$.
Then the \define{$(\pm)$-derivative of $T$ relative to $L_0$} is
$$d_{L_0}T = d_0T = (V(T), C(T)\cup L_0, \omega(d_0 T))$$
where $\omega(d_0 T)$ on a $C(T)\cup L_0$-interval $I$ is defined by any
$a\in A(T)$ in $I$ as above, or if none exist, $\omega(d_0 T)$ is defined
to match the orientation of $T$ on $I$.
Define $\widetilde{A}(T) = A(T) - L_0$.
  \begin{prop}[2.1.2 from \cite{GordonLuecke89}]
   \label{prop:GL2.1.2}
   Let $D$ be any composition of $d^+$'s and $d^-$'s, and $D_0$ the
   corresponding composition of $d_0^+$'s and $d_0^-$'s. Then
   $$C(DT)\subset C(D_0T), \, and \, \widetilde{A}(DT) \supset
   \widetilde{A}(D_0T)$$
  \end{prop}

A \define{graph with dual orientation} is a pair
$\Gamma = (G(\Gamma), \omega(\Gamma))$ where $G(\Gamma)$ is a subgraph
of a fat vertex graph, and $\omega(\Gamma)$ is an assignment of dual
orientation to each corner of $G(\Gamma)$. Note that we will sometimes
consider $G_i$ to simply have a dual orientation, rather than discussing
$\Gamma$.

Given $G_P$ and $G_Q$, a star $T$ such that $L(T)\subset V_Q$ generates
a dual orientation on $G_P(L(T))$ as follows. Vertices of the same sign as
$V(T)$ are given dual orientations of $\omega(T)$, while vertices of
the opposite sign are given dual orientations of $-\omega(T)$.
If an $L(T)$-interval on $V(T)$ corresponds to a corner in $G_P(L(T))$
which has edges in its interior, each subcorner in $G_P(L(T))$ is given
the dual orientation inherited from the corresponding $L(T)$-interval on
$V(T)$. Note that the definition of derivative(s) of a star extends
naturally to derivative(s) of a graph with dual orientation.
We define $\delta \Gamma$ to be the graph with dual orientation obtained
by taking the derivative $d$ at each fat vertex, and for a
subgraph $G_0$ of $\Gamma$, we define $\delta_0\Gamma = \delta_{G_0}\Gamma$
to be the graph with dual orientation obtained by taking the derivative 
$d_{L(G_0)\vert v}$ relative
to $L(G_0)\vert v$ at each fat vertex, where $L(G_0)\vert v$ is the
set of labels at $v$ which meet edges of $G_0$.

  \begin{lem}[2.2.2 from \cite{GordonLuecke89}]
   \label{lem:GL2.2.2}
   If the exceptional labels of $G(L_0)$ are contained in $C(X)$, then
   $\delta \Gamma (T) = \Gamma(d_0 T)$. (Here $d_0 = d_{L_0}$,
   $\delta_0 = \delta_{G(L_0)}$.)
  \end{lem}

If $\Gamma$ is a graph with dual orientation, we define the
\define{dual graph} $\Gamma^*$ (note that this definition is from
\cite{GordonLuecke89} and is not the standard definition
of the dual of a graph).
For each disk face $F$ of $\Gamma$, $\Gamma^*$
has a dual vertex in $int\, F$. The vertices of $\Gamma^*$ are the dual
vertices \emph{and} the vertices of $\Gamma$ (i.e. the fat vertices, though
they are treated as regular vertices in $\Gamma^*$). For each corner $X$ of
the disk face $F$ of $\Gamma$, place an edge $e$ from the (fat) vertex
adjacent to $X$ to the dual vertex corresponding to $F$, and orient $e$
to match the dual orientation of $X$ in $\Gamma$. When this is done to
each corner of each disk face, the directed graph $\Gamma^*$ is obtained.

  \begin{lem}[2.4.1 from \cite{GordonLuecke89}]
   \label{lem:GL2.4.1}
   If $(\delta_0\Gamma)^*$ has a sink or source at a dual vertex
   corresponding to a face $E$ of $\delta_0\Gamma$, then $\Gamma^*$
   has a sink or source at
   a dual vertex corresponding to a face of $\Gamma$ contained in $E$.
  \end{lem}

\begin{cor}[2.4.2 from \cite{GordonLuecke89}]
 \label{cor:GL2.4.2}
 Let $\tau$ be an $L$-type and $T$ a star with $[T] = \tau$. If
 $G(C(T))$ represents $[dT]$, then $G(L)$ represents $\tau$.
\end{cor}

\begin{defn} 
 Let $\tau$ be a nontrivial $L$-type.
 Let $(T_1,\ldots,T_n)$ be a sequence of stars such that
 \begin{enumerate}
 \item $[T_1] = \tau$, $[T_i]$ is nontrivial for $1\leq i \leq n$;
 \item $T_i = d_iT_{i-1}$, where $d_i = d^{\pm}$, for $2\leq i \leq n$;
 \item $[T_n]$ is coherent.
\end{enumerate}   
 Then the sequence $(T_1,\ldots,T_n)$ is a \define{sequence of coherence} for $\tau$.
 \end{defn}
  
 \begin{prop}
 Any nontrivial $L$-type $\tau$ has a sequence of coherence.
 \end{prop}
 \begin{proof}
 
 Let $T_1$ be one of the two stars with sign $V(T_1) = +$ and $[T_1] = \tau$.
 If $\tau$ is coherent, let $n=1$.
 If not, we may assume that not all elements of $A(T_1)$ have the same sign
 (replacing $T_1$ with $\overline{T_1}$ if necessary).
 Let $m$ be the smallest positive integer such that $A((d^+)^mT_1)$ has uniform sign.


 If $C((d^+)^mT_1)$ also has uniform sign, let $n=m+1$, $T_i = d^+ T_{i-1}$,
 $2\leq i \leq n$.
 Otherwise, let $n = m + 2$, with
  \begin{enumerate}
  \item $T_i = d^+ T_{i - 1}$, for $2\leq i \leq m$,
  \item $T_{m + 1} = d^- T_m \, (= \overline{(d^+)^m T_1})$,
  \item $T_{m + 2} = d^+ T_{m + 1}$.
 \end{enumerate}
 
  Clearly all $A(T_n)$ elements have the same sign.
 If $n = 1$ or $n = m + 1$, it is immediate that $[T_n]$ is coherent.
 If $n = m + 2$, then $T_{m + 1}= d^- T_m = \overline{(d^+)^m T_1}$.
 Thus $C(T_{m + 1}) = A((d^+)^m T_1)$ is of uniform sign by definition of $m$.
 Since $n = m + 2$ only when $[(d^+)^m T_1]$ is not coherent, $[T_{m + 1}]$ is
 not coherent.
 But $T_n = T_{m + 2}$ is coherent.
 Note that since $[T_i]$ is coherent only for $i = n$, $[T_i]$ is nontrivial for all $i \leq n$.
 
 The sequence $(T_1,\ldots,T_n)$ is thus a sequence of coherence for $\tau$.
 \end{proof}

\subsection{Index}

\begin{figure}[h]
 \centering
  \subfloat[][A negative edge.]{\includegraphics[width=0.4\textwidth]{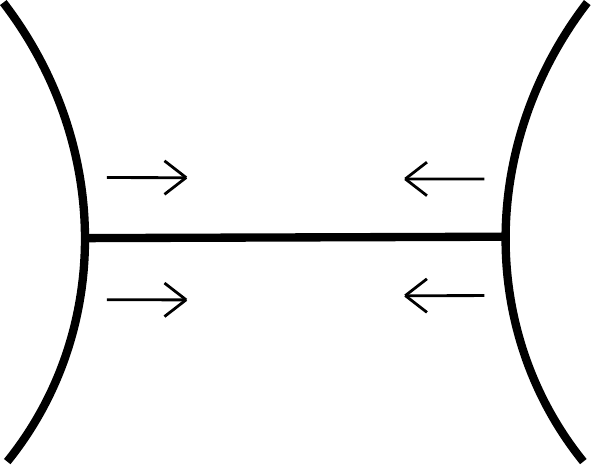}
                 \label{fig:NegativeEdge}}
  \hspace{1cm}
  \subfloat[][A switch edge.]{\includegraphics[width=0.45\textwidth]{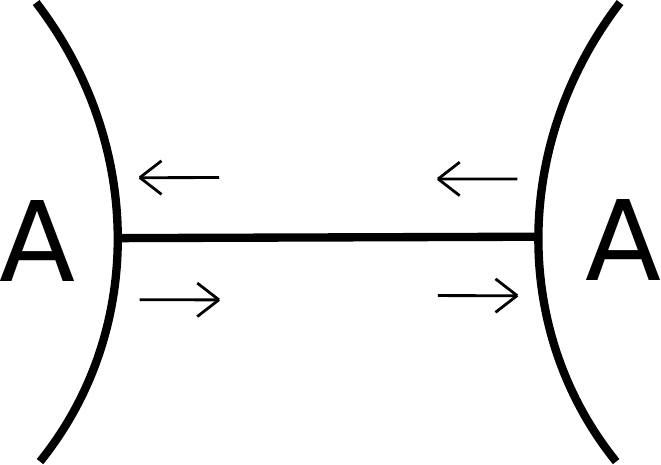}
                 \label{fig:SwitchEdge}}
  \caption[A negative edge and a switch edge.]{}
  \label{fig:EdgeOrientations}
\end{figure}
A \define{negative edge} (Figure \ref{fig:NegativeEdge}) is an edge such
that all four dual orientations
adjacent to the edge are identical (for example, all point out of the
two vertices).
A \define{switch edge} (Figure \ref{fig:SwitchEdge}, sometimes referred
to as a positive edge, particularly in \cite{GordonLuecke89}) is an edge
such that both ends meet vertices at switch labels of the same orientation
(i.e., both anticlockwise switches or both clockwise switches).

 \begin{figure}[h]
 \centering
  \subfloat[][$ind(e) = -1$.]{\includegraphics[width=0.4\textwidth]{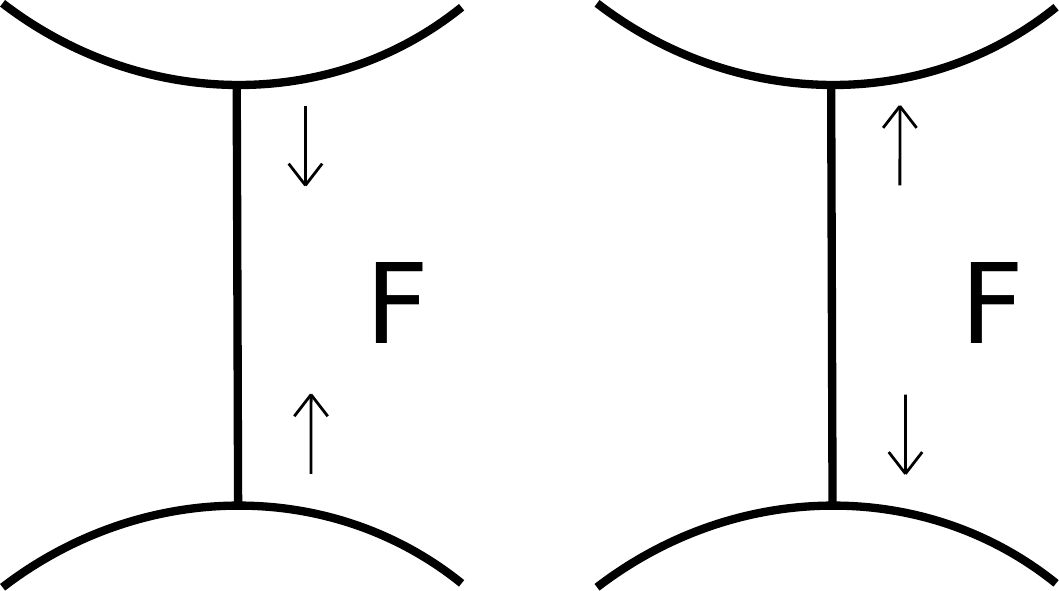}
                 \label{fig:NegativeIndexEdge}}
  \hspace{1cm}
  \subfloat[][$ind(e) = 0$.]{\includegraphics[width=0.4\textwidth]{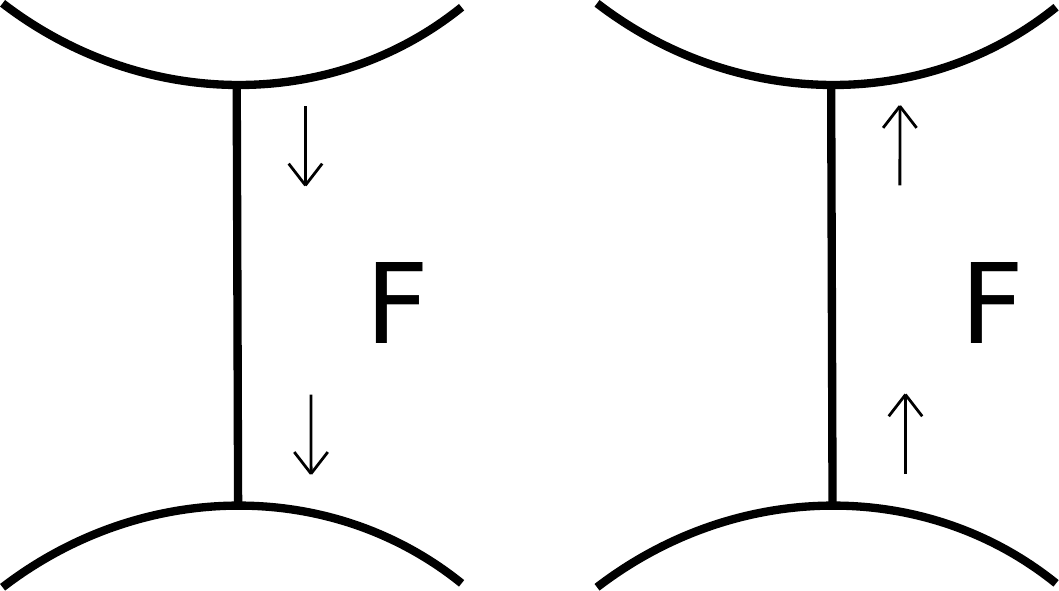}
                 \label{fig:ZeroIndexEdge}}
  \label{fig:EdgeIndex}
  \caption[Edge index calculations.]{}
\end{figure}
A \define{corner} $X$ is $(V(X),I(X),L(X),\omega(X))$ where $V(X)$ is a
model fat vertex, $I(X)$ is an interval on $V(X)$, $L(X)$ a subset of labels
in $I(X)$, and $\omega(X)$ a dual orientation on the $L(X)$-intervals in
$I(X)$. We define the \define{index} of a corner $ind(X) = 1 - s(X)$ where
$s(X)$ is the number of switches in the corner.
For an edge $e\subset \bdd D$, $ind(e) = -1$ if the dual orientations on
the corners of $D$ adjacent to $e$ agree (i.e. are both out or both in,
Figure \ref{fig:NegativeIndexEdge}), while $ind(e) = 0$ if the dual
orientations disagree (Figure \ref{fig:ZeroIndexEdge}).
Finally, for a disk face $D$ of a subgraph of $G_i$, define
$$index\, \bdd D = \sum_{X\text{ a corner of }F} ind(X) +
                  \sum_{e\subset \bdd F} ind(e).$$
Define
\begin{enumerate}
 \item $i = \frac{\vert S(T)\vert}{2} - 1$,
 \item $u$ to be the number of negative edges of $G_P(L)$,
 \item $r$ to be the number of disk faces of $G_P(L)$ representing $T$,
 \item $s$ to be the number of switch edges in $G_P(L)$.
\end{enumerate}

Let $\mathbb{G}$ be a directed graph. For a vertex $v$ of $\mathbb{G}$,
we define $s(v)$ to be the number of times the orientation switches on
edges leaving $v$ (see Figure \ref{fig:VertexSwitches}).
The \define{index} of a vertex $v$ is $I(v) = 1 - \frac{s(v)}{2}$.
The \define{index} of a face is $I(F) = \chi(F) - \frac{s(F)}{2}$,
where $s(F)$ is the number of switches around $F$ (see
Figure \ref{fig:FaceSwitches}).
When $\mathbb{G}$ is actually the dual graph $\Gamma^*$,
define $t$ = $r - \sum I(v)$,
where the sum is taken over all dual vertices of $\Gamma^*$.
\begin{figure}[h]
 \centering
  \subfloat[][$s(v) = 4$, $I(v) = -1$.]{\includegraphics[width=0.4\textwidth]{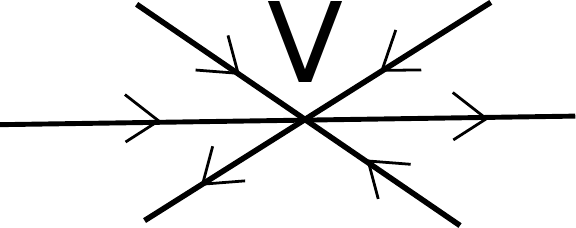}
                 \label{fig:VertexSwitches}}
                 \hspace{1cm}
  \subfloat[][$s(F) = 4$, $I(F) = -1$.]{\includegraphics[width=0.4\textwidth]{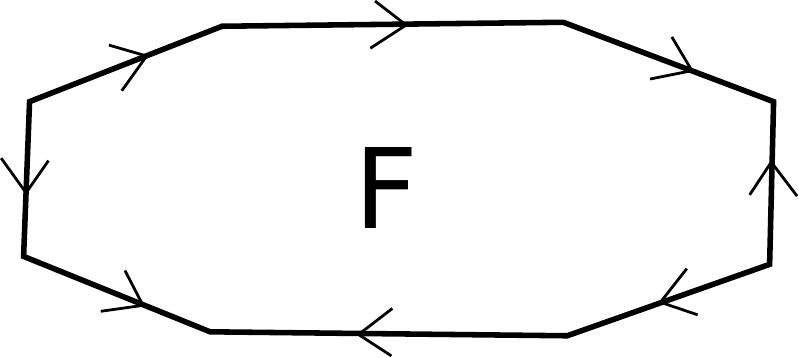}
                 \label{fig:FaceSwitches}}
  \caption{Directed graph index calculations.}
  \label{fig:SwitchCounting}
\end{figure}

Let $L$ be a set of labels with $\vert L \vert \geq 2$ and $\tau$ a nontrivial
$L$-type. Let $T$ be a star with $L(T)=L$ and $[T]=\tau$.
\begin{lem}[\cite{Hoffman95}, 2.3.1]
\label{lem:HoffmanLemma}
Suppose that
\begin{enumerate}[(i)]
 \item all elements of $C(T)$ have the same parity;
 \item all elements of $A(T)$ have the same parity;
 \item $G_P(L)$ does not represent $\tau$.
\end{enumerate}
Then either
\begin{enumerate}[(1)]
\item there exists a new $x$-cycle $\Sigma$ in $G_Q$ such that the
vertices of $\Sigma$ are a subset of either $C(T)$ or $A(T)$; or
\item there is no new $x$-cycle as described in (1), and the following
conditions hold:
      \begin{enumerate}[(a)]
      \item $G_P(L)$ is a connected graph;
      \item $t = u = 0$;
      \item each special vertex of $G_P(L)$ is adjacent to $\vert S(T) \vert$
            switch edges;
      \item each regular vertex of $G_P(L)$ is adjacent to exactly
      $\vert C(T) \vert - 1$ clockwise switch edges and exactly
      $\vert A(T) \vert - 1$ anticlockwise switch edges;
      \item there exists $1,2$-Scharlemann cycles in $G_Q$ on $C(T)$ and $A(T)$;
      \item $n\leq \vert C(T) \vert = \vert A(T) \vert$, where $n$ is the order
      		of the Scharlemann cycle;
      \item $\vert [C(T), L \setminus C(T)] \vert =
             \vert [A(T), L \setminus A(T)] \vert = p - 2$; and
      \item the vertices of $C(T)$ (also $A(T)$) are connected in $G_Q$.
      \end{enumerate}
\end{enumerate}
\end{lem}

For a directed graph $\mathbb{G}$,
\begin{lem}[\cite{GordonLuecke89}, 2.3.1]
\label{lem:EulerCharLemma}
$\sum_{vertices} I(v) + \sum_{faces} I(F) = 2$.
\end{lem}

\begin{lem}[Extension of \cite{GordonLuecke89}, 2.3.3]
\label{lem:GLIndexLemma}
Let $\Gamma$ be a graph with dual orientation, $F$ a disk face of a
subgraph of $\Gamma$ such that $ind_{\Gamma}(\bdd F) \leq 0$.
Then $\Gamma$ has one of the following in $F$:
\begin{enumerate}
\item A switch edge;
\item A dual source or sink face;
\item A fat vertex with uniform dual orientation.
\end{enumerate}
\end{lem}
\begin{proof}
\emph{This proof is adopted from Gordon and Luecke's proof.}
Let $2F$ be the double of $F$, i.e. $2F = F \cup_{\bdd F} -F$. Let
$2\Gamma^*$ denote the double of $\Gamma^*\subset 2F$. By Lemma
\ref{lem:EulerCharLemma},
$$ 2\sum_{v\in \Gamma^*\cap F} I(v) +
   2\sum_{f \text{ face of }\Gamma^*\cap F} I(f) +
   \sum_{X \text{ a corner of F}}ind(X) +
   \sum_{e\subset \bdd F} ind(e) = 2 $$
Therefore
$$ 2\sum I(v) + 2\sum I(f) + ind_{\Gamma}(\bdd F) = 2. $$
Thus $ \sum I(v) + \sum I(f)  > 0$.
Note that a positive index vertex of $\Gamma^*$ implies the presence
of either a fat vertex sink or source or a dual sink or source face,
while a positive index face of $\Gamma^*$ corresponds to a switch edge
of $\Gamma$.
\end{proof}

 Lemma \ref{lem:CircuitImpliesRepsTrivialType} and Lemma \ref{lem:NoCycles}
 show that the existence of certain types of cycles in $G_P(V)$ imply that
 $G_P(V)$ represents certain $V$-types. Note that Lemma \ref{lem:NoCycles}
 can cover the trivial type, but Lemma \ref{lem:CircuitImpliesRepsTrivialType}
 is provided since in the case of the trivial type, the result can be
 obtained in a slightly more general setting.
 
 Theorems \ref{thm:TrivialTypeSDisks} and \ref{thm:HoffmanTree} show the
 existence of trees in $G_P$ when particular $V$-types are not represented
 by $G_P(V)$.
 
 Lemma \ref{lem:NoCycles} and Theorem \ref{thm:HoffmanTree} are an extension
 of Theorem 2.4.2 from \cite{Hoffman95}.

\begin{lem}
 \label{lem:CircuitImpliesRepsTrivialType}
 Let $V$ be any set of vertices in $G_Q$.
 Suppose $G_P$ has $\mathbb{P}(V_r,V)$.
 If $G_P(V)$ has a cycle on negative (or positive) vertices, then $G_P(V)$
 represents the trivial $V$-type.
\end{lem}
\begin{proof}
 Since $G_P$ has $\mathbb{P}(V_r,V)$, every regular vertex in $G_P$ is
 adjacent to a parallel vertex at a label in $V$.
 Let $\sigma$ be an innermost cycle on vertices of uniform sign in
 $G_P(V)$. Without loss of generality, suppose $\sigma$ is negative.
 Since $G_Q$ contains a Scharlemann cycle, $\sigma$ cannot separate the
 special vertices from one another.
 Thus $\sigma$ bounds a unique disk $D\subset \Delta$ for some
 $1,2$-bigon face $\Delta$.
 Suppose $D$ contains a positive vertex.
 Every positive vertex $x$ in $D$ has an edge $e_x$ meeting $x$ at a
 label in
 $V$, such that $e_x$ connects $x$ to another positive vertex. By the parity
 rule, there cannot be an edge from $x$ to another positive vertex $y$
 meeting both vertices at labels in $V$. Thus by Proposition
 \ref{prop:AsManyEdgesAsVertsImpliesCircuit}, there is
 a cycle among the positive vertices in $D$, which contradicts the choice of
 $\sigma$.
 Thus all vertices in $D$ are negative.
 Since $\sigma$ is a negative cycle in $G_P(V)$ that bounds a disk with
 only negative vertices, $G_P(V)$ represents the trivial $V$-type.
\end{proof}

 

 \begin{lem}
 \label{lem:NoCycles}
 Suppose $V\subset V_Q$ has uniform sign and uniform dual orientation type,
 that is, $V$ is entirely contained in $A$, $C$, $B_+$ or $B_-$.
 Let $W\subset V_Q$ be a set of opposite sign from $V$.
 Assume that $G_Q[V,W]$ has $\mathbb{A}(V,V_r)$.
 If $G_P$ has an oriented cycle where each edge has its tail at a label
 in $V$ and its head at a label in $W$, then $G_P$ represents $\tau$.
 \end{lem}
 \begin{proof}

 Suppose some component of $G_P[V,W]$ has such an oriented cycle $\sigma$.
 Since the special vertices are connected by edges from the Scharlemann
 cycles in $G_Q$, the special vertices must be in the same $\sigma$-disk.
 Assume $\sigma$ is innermost in the $\sigma$-disk $D$ not containing the
 special vertices, and assume $\sigma$ is positive.
 Let $V_D$ be the $\sigma$-set in $D$, and suppose that
 $V_D$ contains a vertex antiparallel to $\sigma$. Since $G_Q[V,W]$
 has $\mathbb{A}(V,V_r)$, $G_P[V,W]$ has $\mathbb{P}(V_r, V)$. Thus
 for each negative vertex $x\in V_D$, there is an edge $e_x$ which meets
 $x$ at a label in $V$ and meets another negative vertex in $V_D$ at a
 label in $W$. Thus there must be an oriented cycle of such edges,
 contradicting the innermost assumption of $\sigma$.

 So $V_D$ has no negative vertices. Since every vertex is adjacent
 to an antiparallel vertex via a switch edge, $V_D = \emptyset$.
 Let $x_0$ be a vertex on $\bdd D$, and let $e_0$ be the edge leaving
 $x_0$ at a $V$ label. Without loss of generality we assume that the
 dual orientation of the corner in $D$ adjacent to $e$
 is into $x_0$. Let $X_0$ be the corner of $D$ at $x_0$.
\begin{figure}[h]
 \centering
  \subfloat[][$ind(e_x) = 0$, $ind(X)\leq 0$.]{\includegraphics[width=0.4\textwidth]{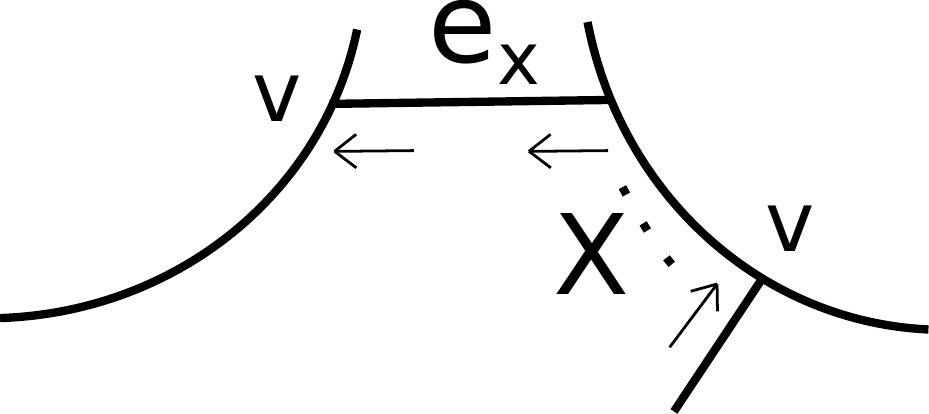}
                 \label{fig:CalculatingDegree1}}
  \hspace{1cm}
  \subfloat[][$ind(e_x) = -1$, $ind(X)\leq 1$.]{\includegraphics[width=0.4\textwidth]{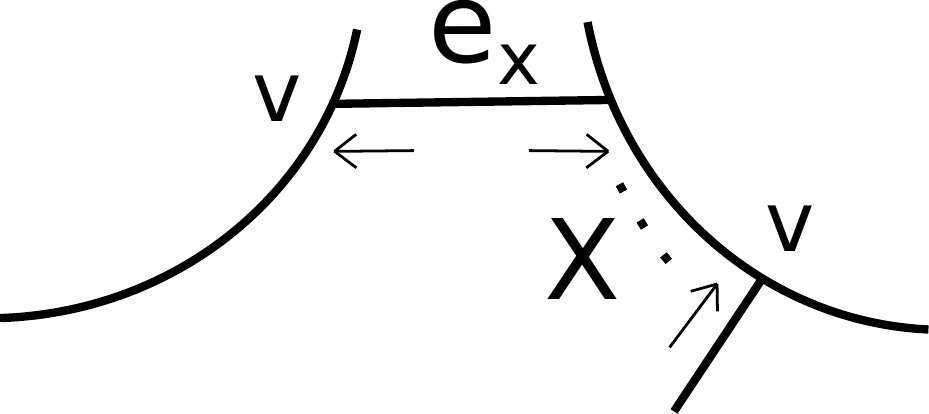}
                 \label{fig:CalculatingDegree2}}
  \caption{Index calculations.}
  \label{fig:CalculatingDegree}
\end{figure}

 Now note that if $\text{ind } e_X = 0$, then $\text{ind } X \leq 0$.
 If $ind(e_X)= -1$, then $ind(X) \leq 1$
 (Figure \ref{fig:CalculatingDegree}).
 Either way, $ind(e_X) + ind(X) \leq 0$. Therefore
 $index(\bdd D)\leq 0$. By Lemma \ref{lem:GLIndexLemma},
 there is a face in $G_P(L)$ in $D$ that represents $\tau$.
 \end{proof}

\begin{thm}
\label{thm:TrivialTypeSDisks}
 Let $V$ be an $(s)$-set contained in an $(s)$-disk. Let $W$ be
 the set of vertices in $G_Q$ of sign $-s$ to which $V$ have edges.
 If $G_P(V)$ does not represent the trivial $V$-type, the following all 
 hold:
 \begin{enumerate}
  \item \label{itm:RegVertsLeave} $V^* = V_r$;
  \item \label{itm:NEdgesLeave} $\vert [V,W] \vert = p - 2$;
  \item \label{itm:SchCycleExistence} $V$ has a Scharlemann cycle;
  \item \label{itm:TreeExistence} $G_P([V,W])$ consists entirely of a
  $(V\to W\to 1)$ tree and a $(V\to W\to 2)$ tree.
 \end{enumerate}

\end{thm}

\begin{proof}

 By Lemma \ref{lem:SDiskProperties}, $G_P$ has $P(V_r,V)$ and $V^* \supset V_r$,
 implying that $[V,W]\geq p - 2$.
 By Lemma \ref{lem:CircuitImpliesRepsTrivialType}, $G_P(V)$ cannot have a cycle
 on parallel vertices, so by Lemma \ref{lem:AntiparallelEdgesInQGiveCircuitInP},
 $[V,W] = p - 2$. This means that $V^* = V_r$, so by Lemma
 \ref{lem:SDiskProperties} $V$ has a Scharlemann cycle.

 Finally, Lemma \ref{lem:APropertyImpliesCircuitsOrTrees} gives the desired
 trees.
\end{proof}

\begin{thm}[Extension of \cite{Hoffman95} 2.4.2]
\label{thm:HoffmanTree}
Suppose $V\subset V_Q$ has uniform sign and uniform dual orientation type,
that is, $V$ is entirely contained in $A$, $C$, $B_+$ or $B_-$.
Let $W\subset V_Q$ be a set of opposite sign from $V$.
If $G_Q[V,W]$ has $\mathbb{A}(V,V_r)$, then $G_P$ has a negative
$(V\to W\to 1)$ tree and a positive $(V\to W\to 2)$ tree.
\end{thm}

\begin{proof}
Any component of $G([V,W])$ that is not in a $(V\to W\to r)$ tree with $r$ a
special vertex must have an oriented cycle such that each edge has its
tail at a $V$ label and its head at a $W$ label.
\end{proof}

%% file: sections/GLProofRewrite.tex
\section{Gordon-Luecke Proof Deconstruction}

In this section, the proof from \cite{GordonLuecke89} is deconstructed.
\emph{This is not a new proof.} It is merely the proof of Gordon
and Luecke rewritten so as to more easily extract lemmas which will be
used later.

\subsection{Good, Bad, and Ugly Corners}
\label{ssec:GLSetup}
  
  A corner $X$ is \define{ugly} if there exist $A(X)$ elements with
  differing parities. Since $A(X) = A(-X)$, $X$ is ugly if and only
  if $-X$ is ugly.
  If $X$ is not ugly, define $char\, A(X) = char(a,V(X))$ for any
  $a\in A(X)$.  
  Choose a clockwise character $\eta_c = \pm$ and an anticlockwise character
  $\eta_a = \pm$.
  Two-color the faces of $G_P$. Choose
  the B/W coloring so that a pair $(l,v)$ of character $\eta_c$ is WB
  (going counterclockwise). We refer to $(l,v)$ pairs as WB or BW.
  For $C_1,C_2\in \{B,W\}$, a corner $X$ is $C_1C_2$ if the leftmost label
  of $X\, (\in \bdd I(X))$ is $C_2C_1$, and the rightmost label is $C_1C_2$.
  For example, in Figure \ref{fig:GoodAtoms}, the first corner is $BB$,
  the second $WW$, and the third and fourth $BW$.
  
  \begin{figure}[h]
 \centering
 \includegraphics[width=0.8\textwidth]{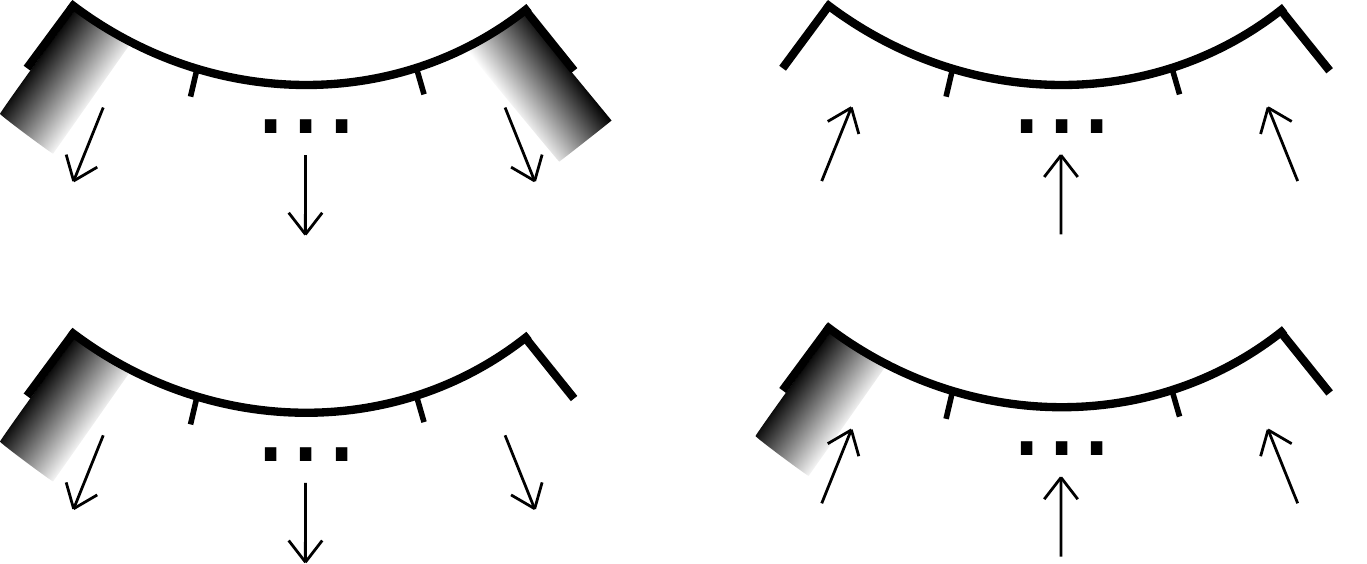}
 \caption{Good atoms.}
 \label{fig:GoodAtoms}
\end{figure}
  An \define{atom} is a corner with no switch labels.
  The atoms in Figure \ref{fig:GoodAtoms}
  are \define{good}, and all others are \define{bad}. Note that
  \begin{enumerate}
   \item an atom $X$ is good if and only if $-X$ is bad, and
   \item atoms on either side of a clockwise switch are both good
   if $char(l,v) = \eta_c$; otherwise, both atoms are bad.
  \end{enumerate}
  If $char\, A(X) = \eta_a$, then $X$ is \define{good} if and
  only if all maximal atoms in $X$ are good. If $char\, A(X) = -\eta_a$,
  $X$ is \define{good} if and only if \emph{some} maximal atom in $X$
  is good.
  $X$ is \define{bad} if it is neither good nor ugly.
  A clockwise (anticlockwise) switch is \define{double-sided} if it has
  character $\eta_c$ ($\eta_a$) and \define{single-sided} if it has
  character $-\eta_c$ ($-\eta_a$).

\begin{figure}[h]
 \centering
 \includegraphics[width=0.5\textwidth]{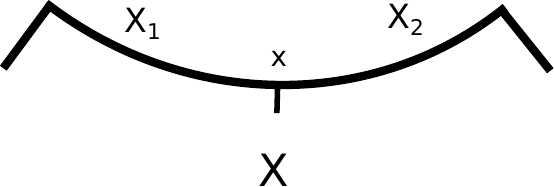}
 \caption{Corner for Lemma \ref{lem:GL2.6.1}.}
 \label{fig:GL2-6-1}
\end{figure}
    
  \begin{lem}[2.6.1 from \cite{GordonLuecke89}]
  \label{lem:GL2.6.1}
  Let $x\in S(X)$, and $X$, $X_1$, and $X_2$ be as Figure \ref{fig:GL2-6-1}.
  Assume $X$ is not ugly.
  \begin{enumerate}[(i)]
  \item If $x$ is double-sided then $X$ is good if and only if $X_1$
  and $X_2$ are good.
  \item If $x$ is single-sided then $X$ is good if and only if $X_1$
  or $X_2$ is good.
  \end{enumerate}
  \end{lem}
  
  \begin{lem}[2.6.2 from \cite{GordonLuecke89}]
  \label{lem:GL2.6.2}
  If $char \, A(X) = -\eta_a$ and there exists $c\in C(X)$ with
  $char \, c = \eta_c$ then $X$ is good.
  \end{lem}
  
  \begin{lem}[2.7.1 from \cite{GordonLuecke89}]
  \label{lem:GL2.7.1}
  Let $F$ be a disk face of a subgraph of $\Gamma$ such that each corner
  of $F$ with respect to $\Gamma$ is good. Then $ind_\Gamma \bdd F \leq 0$.
  \end{lem}

 \begin{defn}
 Let $\tau$ be a nontrivial $L$-type, and let $(T_1,\ldots,T_n)$ be a
 sequence of coherence
 for $\tau$.
 Let $L_0\subset L$, and let $(R_1,\ldots,R_n)$ be the sequence obtained by
 letting $R_1 = T_1$, $R_i = d_0^{\pm}R_{i-1}$, where the sign of $d_0^\pm$
 is chosen to match that of the corresponding $d^{\pm}$.
 Then $(R_1,\ldots ,R_n)$ is a \define{$L_0$-sequence} for $\tau$.
 \end{defn}
 
 Note that by Proposition \ref{prop:GL2.1.2}, $C(T_n)\subset C(R_n)$ and
 $\widetilde{A}(T_n) \supset \widetilde{A}(R_n)$.
 Let $I$ be an $L_0$-interval at $V_+$.
 
 \begin{prop}
 The corner $R_n\vert I$ is not ugly.
 \end{prop}
 \begin{proof}
  $\widetilde{A}(R_n) \subset \widetilde{A}(T_n) \subset A(T_n)$.
  Since $T_n$ is coherent, $\widetilde{A}(R_n)$ is of uniform sign.
 \end{proof}
 
 Let $\eta_c = \text{char}(C(T_n),V(T_n))$ and
 $\eta_a = -\text{char}(\widetilde{A}(R_n),V(R_n))$.
 If $\widetilde{A}(R_n)$, $\eta_a$ can be chosen arbitrarily.
 Recall that $V(T_n) = V(R_n) = V_+$.
 For each $L_0$-interval $I$, define $\epsilon(I) = \pm$ by requiring
 $\epsilon(I)(R_n\vert I)$ to be good.
 Define the $L_0$-type $\tau_0$ by $\tau_0 = (\epsilon(I) : I$ an $L_0$-interval$)$.
 We call $\tau_0$ an \define{inherited $L_0$-type} of $\tau$.
 Note that we have not shown, nor do we use, that $\tau_0$ is
 independent of the coherence sequence chosen.

\subsection{Proof of Theorem \ref{thm:SchCycleExistence}}

For the duration of this section, let $D$ be a disk in $Q$ that is
either the complement of a small open disk disjoint from $G_Q$, or a
disk bounded by a new $x$-cycle $\Sigma$ in $G_Q$.
Let $L$ be the set of vertices of $G_Q$ in int $D$.
Note that $\vert L \vert \geq 2$, as there can be no new great $x$-cycles,
nor isolated vertices, in $G_Q$.

\begin{lem}
 \label{lem:NontrivialBaseCase}
 Let $\tau$ be a nontrivial $L$-type with sequence of coherence
 $(T_1,\ldots,T_n)$. If $G_Q\cap D$ contains no $x$-cycle on vertices of
 $A(T_n)$ or $C(T_n)$, then $G_P(L)$ represents $\tau$.
\end{lem}

\begin{proof}

 If $G(L_n)$ does not represent $[T_n]$, then by Lemma
 \ref{lem:HoffmanLemma}, $G_Q$ contains
 an $x_0$-cycle $\Sigma_0$ with vertices either in $A(T_n)$ or $C(T_n)$,
 contradicting the hypothesis.
 Thus $G_P(L_n)$ represents $[T_n]$, so $G_P(L)$ represents $\tau$
 by Corollary \ref{cor:GL2.4.2}.
 
\end{proof}

 \begin{lem}
 \label{lem:TrivialBaseCase}
 Suppose that $G_Q\cap D$ contains no $x$-cycle.
 Then $G_P(L)$ represents the trivial $L$-type.
\end{lem}

\begin{proof}
 Let $J\subset L$ be all vertices of opposite sign from $\Sigma$ (or, if
 $\Sigma = \emptyset$, choose $J$ to be the positive elements of $L$).
 $\Sigma$ cannot be a great new $x$-cycle by Theorem 
 \ref{thm:NoNewGreatXCycle}, so $J\neq \emptyset$.
 
 Suppose first that for some vertex $x_0$ in $G_P$, every label $y\in J$
 on $x_0$ is adjacent to an edge leading to an antiparallel vertex.
 These edges must all lead to parallel labels, implying that the
 labels on both ends of each edge are in $J$.
 We thus have an $x_0$-cycle in $D$, contrary to our assumption.
 
 Therefore for each vertex $x$ in $G_P$ there exists a label $y(x)\in J$
 such that the edge $e(x)$ leaving vertex $x$ at label $y$ goes to a
 vertex parallel to $x$.
 Note that the label at the other end of $e(x)$ must have opposite sign
 from $y(x)$ by the parity rule.
 If $E = \{e(x)\}$, $G_P(E)$ will have circuits on every connected
 component.
 Every connected component will have uniform sign (since edges
 $e(x)$ connect parallel vertices).
 An innermost circuit on parallel vertices will therefore bound a disk
 $E$ with only vertices parallel to the circuit.
 $E$ thus contains a disk face representing the trivial $L$-type.
\end{proof}

\begin{lem}
\label{lem:TrivialInduction}
 Suppose $L$ is an arbitrary set of vertices of $G_Q$ and $L'\subset L$.
 If $G_P(L')$ represents the trivial $L'$-type, then $G_P(L)$ represents the
 trivial $L$-type.
\end{lem}
\begin{proof}
 $G_P(L')$ represents the trivial $L'$-type, and thus has a disk face $E'$
 representing the trivial $L'$-type.
 Any face $E$ of $G(L)$ in the disk $E'$ thus represents the trivial
 $L$-type.
\end{proof}

\begin{lem}
\label{lem:NontrivialInduction}
 Let $\tau$ be a nontrivial $L$-type, with sequence of coherence
 $(T_1,\ldots,T_n)$. Suppose $G_Q\cap D$ contains a cycle
 $\Sigma_0$ on vertices of $A(T_n)$ (respectively $C(T_n)$), such that the
 $\Sigma_0$-disk $D_0\subset D$ contains vertices which are not in
 $A(T_n)$ (respectively $C(T_n)$).
 Let $L_0$ be the $\Sigma_0$-set contained in $D_0$.
 Suppose $G_P(L_0)$ represents $\tau_0$, an inherited $L_0$-type
 of $\tau$ obtained from the above sequence of coherence.
 Then $G_P(L)$ represents $\tau$.
 \end{lem}
 
 \begin{proof}
 We assume that the vertices of $\Sigma_0$ are in $C(T_n)$ by possibly replacing
 $T_n$ with $\overline{T_n}$. This can be achieved by replacing $d_n$ by its
 negative in the sequence of coherence, and replacing $\tau_0$ with $-\tau_0$.
 Let $E$ be a face of $G_P(L_0)$ that represents $\tau_0$. Then there exists
 $\eta = \pm$ such that if a corner of $E$ at a vertex $v$ is in the
 $L_0$-interval $I$, then sign $v$ $= \eta \epsilon(I)$ (recall the definition
 of $\eta$ from the end of section \ref{ssec:GLSetup}). Let $J$ be the
 subinterval of $I$ corresponding to the corner, and let $(R_1,\ldots,R_n)$ be
 the $L_0$-sequence for $\tau$ obtained from $(T_1,\ldots,T_n)$.
   
  \emph{The remainder of the proof is the same as the end of the proof of
  Lemma 2.8.2 in \cite{GordonLuecke89}, but is included here for completeness.}
  
 \begin{claim}
 \label{CLA:GoodCorner}
 $\eta(\text{sign }v)R_n\vert J$ is good.
 \end{claim}
 \begin{proof}
  Since $\eta(\text{sign }v) = \epsilon(I)$, we are done if $J$ is an
  $L_0$-interval, by the definition of $\tau_0$. So assume this is not
  the case.
  Then $J$ is a subinterval of an $L_0$-interval $I$ with at least
  one endpoint of $J$ an exceptional label in $G_P(L_0)$.
  Since the exceptional labels are
  contained in $C(T_n)$, $\epsilon(I) = +$ by Lemma \ref{lem:GL2.6.2}.
  Now $R_n\vert J$ is good by Lemma \ref{lem:GL2.6.1}(i).
  \end{proof}

  Let $\Gamma_i = \Gamma(R_i)$, $1\leq i \leq n$, and let
  $\delta_0 = \delta_{G(L_0)}$. Since
  $$\{\text{exceptional labels of }G(L_0)\} \subset
  \{\text{vertices of }\Sigma_0\} \subset C(T_n) \subset
  C(T_i)\subset C(R_i),$$ for $1\leq i \leq n$, we have from
  Lemma \ref{lem:GL2.2.2} that
  $\Gamma_i = \delta_0 \Gamma_{i-1}$, $2 \leq i \leq n$.
  
  By Claim \ref{CLA:GoodCorner}, $\eta X$ is good for any corner $X$
  of $E$ in $\Gamma_n$.  
  A face of $\Gamma^*_n \cap E$ of positive
  index corresponds to a switch-edge $e$ of $\Gamma_n \cap E$. Since the
  endpoints of $e$ have opposite characters, one endpoint of $e$ will be a 
  double-sided switch, the other single-sided. Split $E$ along $e$,
  i.e., let $E = E_1 \cup_e E_2$. By Lemma \ref{lem:GL2.6.1}, for some
  $i\in \{1,2\}$, $\eta X$ is good for every corner $X$ of $E_i$ in
  $\Gamma_n$.
  We repeat until we get a disk $F\subset E$,
  bounded by edges of $\Gamma_n$, such that
  \begin{enumerate}[(a)]
   \item \label{itm:GoodCorners}for each $\Gamma_n$-corner $X$ of $F$,
   $\eta X$ is good;
   \item \label{itm:NoPosIndexFaces} $\Gamma^*_n$ has no faces of positive
   index in $F$.
  \end{enumerate}
  
  Lemma \ref{lem:GL2.7.1} and \ref{itm:GoodCorners} imply that
  $ind_{\Gamma_n}\bdd F \leq 0$.
  Lemma \ref{lem:GLIndexLemma} and \ref{itm:NoPosIndexFaces} imply that
  $\Gamma^*_n$ has a sink or source in $F$.
  Since $\Gamma_i = \delta_0\Gamma_{i - 1}$, $2\leq i \leq n$,
  By Lemma \ref{lem:GL2.4.1}, $\Gamma^*_1 = \Gamma(T_1)^*$ has a
  sink or source at a dual vertex in $F$. Thus $G_P(L)$ represents $\tau$.
  
 \end{proof}

 \begin{thm}
  \label{thm:ScharlemannCycleInNewXCycle}
  $G_P(L)$ represents all $L$-types or there exists a Scharlemann cycle in
  int $D$.
 \end{thm}

 \begin{proof}
  Assume no Scharlemann cycle exists in int $D$. We proceed by induction on
  the number of new $x$-cycles.
  If no new $x$-cycles exist in int $D$, then by Lemmas
  \ref{lem:NontrivialBaseCase} and \ref{lem:TrivialBaseCase} $G_P(L)$
   represents all types.
  
  We now make the inductive assumption that any disks containing fewer new
  $x$-cycles than $D$ represent all types. Suppose that $G_P(L)$ does not
  represent the $L$-type $\tau$. By Lemma \ref{lem:TrivialInduction}, $\tau$ is
  nontrivial. Let $(T_1,\ldots,T_n)$ be a sequence of coherence for $\tau$.
  By Lemma \ref{lem:NontrivialBaseCase}, int $D$ contains an $x$-cycle on
  $A(T_n)$ or $C(T_n)$. By Lemma \ref{lem:NontrivialInduction} and the
  inductive assumption, $G_P(L)$ represents $\tau$.
 \end{proof}
 
 Although the proof of the following theorem is in \cite{Parry90}, the
 explanation of the homological implication is in Section 4 of \cite{Gordon97}.
 \begin{thm}[Consequence of \cite{Parry90}]
 \label{thm:AllTypesTorsion}
 If $G_P$ represents all $V_Q$-types, then
 $H_1(M(\gamma))$ contains a nontrivial torsion element.
 \end{thm}
 
 Theorem \ref{thm:AllTypesTorsion} contradicts $M(\gamma) = S^3$,
 concluding the proof of Theorem \ref{thm:SchCycleExistence}.

%% file: sections/RotationFreeGraph.tex
\section{Rotation-Free Graph}
\label{sec:RotFreeGraph}
Assume $\tau$ is an $L$-type, $T$ a coherent star with $[T] = \tau$.
Let $G_P(L)$ have the dual orientation inherited from $\tau$.
Let $\mathbb{F}$ be a collection of faces of $G_P(L)$. We define
$Rev(G_P(L), \mathbb{F})$ to be the graph obtained by reversing the dual
orientations on each $F\in \mathbb{F}$.
We define a fat vertex graph with dual orientation to be
\define{representative}
if it has a face that is a source or a sink with respect to the dual
orientation.

\begin{lem}
\label{lem:ReversingNotRepresentative}
For any collection of faces $\mathbb{F}$, $G_P(L)$ is representative if
and only if $Rev(G_P(L), \mathbb{F})$ is representative.
\end{lem}
\begin{proof}
 Any source or sink in one graph will correspond to a source or sink in
 the other.
\end{proof}

\begin{prop}
\label{prop:EvenDegreeTwoColor}
 Let $G$ be a graph. If every vertex of $G$ has even degree, the faces
 of $G$ can be two-colored such that each edge has each color on exactly
 one side.
\end{prop}
\begin{proof}
 The statement clearly holds if $G$ has zero edges.
 Given a $G$ satisfying the hypothesis that has more than zero edges, select
 a disk face $D$ of $G$ (possibly containing vertices in its interior) and
 remove all edges in $\bdd D$ to obtain $G'$. $G'$ has strictly fewer edges,
 but still satisfies the hypothesis, proving the statement by induction.
\end{proof}

We assume that $G_P(L)$ is not representative, $[T]$ is nontrivial,
 and that neither $A(T)$
nor $C(T)$ have a new $x$-cycle. Let $E_{switch}$ be
the switch edges of $G_P(L)$. Then by Lemma
\ref{lem:HoffmanLemma} and Proposition \ref{prop:EvenDegreeTwoColor},
$G_P(E_{switch})$ can be two-colored black and white, with every edge
between a white face and a black face. $G_P(L)$ inherits a coloring from
$G_P(E_{switch})$. Let $\mathbb{F}_{black}$ be the black faces of $G_P(L)$.

\begin{lem}
$Rev(G_P(L), \mathbb{F}_{black})$ has no switch edges and is
not representative.
\end{lem}
\begin{proof}
 $G = Rev(G_P(L), \mathbb{F}_{black})$ is not representative by Lemma
 \ref{lem:ReversingNotRepresentative}. Since each switch edge of $G_P(L)$
 has a black face on exactly one side, the dual orientations on exactly
 one side of every switch edge of $G_P(L)$ gets reversed in $G$.
 
 All edges in $G_P(L)$ which are not switch edges are adjacent to two
 faces of the same color, and therefore are not switch edges in $G$.
\end{proof}

We will henceforth define $RF_P = Rev(G_P(L), \mathbb{F}_{black})$.
We will define
\begin{align}
A(RF) &= \{(l, v)\in V_Q \times V_P \vert l\text{ on } v
\text{ is an anticlockwise switch label, in }RF_P\},\\
C(RF) &= \{(l, v)\in V_Q \times V_P \vert l\text{ on } v
\text{ is a clockwise switch label, in }RF_P\}.
\end{align}

\begin{lem}
\label{lem:NoAHCircuits}
$RF_P$ does not have a directed cycle with each tail at a label from
$A(RF)$. Similarly for $C(RF)$.
\end{lem}
\begin{proof}
By Lemma \ref{lem:GLIndexLemma}, any such directed cycle $\sigma$ would 
bound faces each containing a switch edge, a source/sink face, or a
source/sink fat vertex. Since $RF_P$ has no switch edges or source/sink 
faces, each face bounded by $\sigma$ contains a source/sink face.

However by Lemma \ref{lem:HoffmanLemma}, only the special vertices will be
source/sink vertices in $RF_P$. Since the special vertices have edges
between them (from any Scharlemann cycle), they must be on the same side
of $\sigma$, a contradiction.
\end{proof}

The following can be seen as a generalization of Theorem \ref{thm:HoffmanTree}. Note that if $T$ is a trivial type, this section
does not apply, but we get similar trees from Theorem
\ref{thm:TrivialTypeSDisks}.
\begin{cor}
\label{cor:RFTrees}
$RF_P$ has disjoint $(A(RF)\to 1)$ and $(A(RF)\to 2)$ trees. Similarly for
$C(RF)$.
\end{cor}
\begin{proof}
This follows immediately from Lemma \ref{lem:HoffmanLemma} and Lemma
\ref{lem:NoAHCircuits}.
\end{proof}

Suppose now that $\tau$ is a $V_Q$-type.
We can $2$-color the entire graph $G_P$, since $q$ is always even.
Let $\mathbb{F}$ be the black faces. The parity rule allows
us to partition the corners of $T$ into black corners and white corners.
Thus $Rev(G_P(L), \mathbb{F})$ defines a star $\widehat{T}$, the
\define{conjugate} of $T$. The above results
therefore imply that $G_P$ represents $\tau$ if and only if $G_P$
represents $\widehat{\tau} = [\widehat{T}]$.
Although the conjugate is defined up to a choice of coloring,
$G_P$ represents $\widehat{T}$ if and only if $G_P$ represents
$\overline{(\widehat{T})}$, so generally the choice is not specified.
A coherent $V_Q$-star $T$ is
\define{bicoherent} if $\widehat{T}$ is also coherent, or (equivalently)
if $B_+(T)$ and $B_-(T)$ are each of uniform sign.
Note that for some $s\in \{+,-\}$, $A(T) = B_s(\widehat{T})$ and
$C(T) = B_{-s}(\widehat{T})$. We are therefore describing a symmetry
between the $A(T), C(T)$ sets and the $B_+(T), B_-(T)$ sets.
 \begin{figure}[h]
  \centering
  \includegraphics[width=0.8\textwidth]{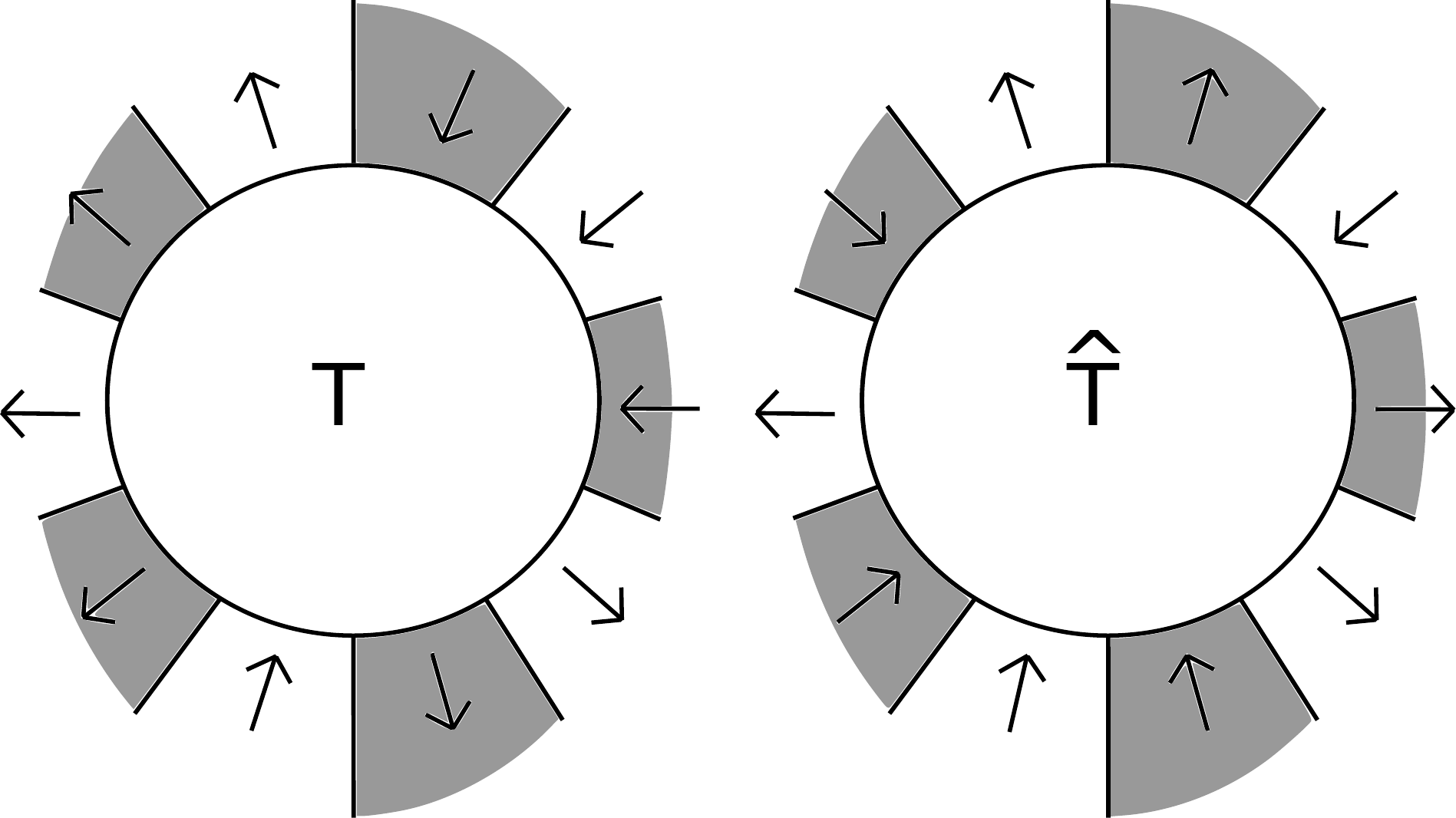}
  \caption{The star $T$ and its conjugate $\widehat{T}$.}
  \label{fig:ConjugateType}
 \end{figure}

\begin{thm}
Let $T$ be bicoherent with $[T]$ a nontrivial $V_Q$-type, such that
$[\widehat{T}]$ is also nontrivial. If $G_P$ has no new
$x$-cycle on $A(T)$ or $C(T)$, then $G_P$ has a new
$x$-cycle on $B_+(T)$ or $B_-(T)$.
\end{thm}
\begin{proof}
Suppose $G_P$ has no new $x$-cycle on $B_+$ or $B_-$. Suppose $v$ is a
leaf in the $(A(RF)\to 1)$ tree which exists by Corollary \ref{cor:RFTrees}.
This means that no
$B_+$ or $B_-$ label on $v$ can be adjacent to an edge going to an
$A(RF)$ label. But $B_+$ and $B_-$ are each of uniform sign, and
$G_Q$ cannot have a new $v$-cycle, so at least one edge leaving
$v$ at a $B_+$ label must meet another
vertex at a $C(RF)$ label, and similarly for $B_-$. Therefore two edges
of the $(C(RF)\to 2)$ tree come into $v$. Hence $v$ is a vertex of the 
$(C(RF)\to 2)$ tree with at least two children.

Since $v$ was an arbitrary leaf of the $(A(RF)\to 1)$ tree, every leaf of
the $(A(RF)\to 1)$ and $(A(RF)\to 2)$ trees must have two children in the
$(C(RF)\to j)$ trees. Thus if the
$(A(RF)\to i)$ trees together have $r$ leaves, the $(C(RF)\to j)$ trees
will together have more than $r$ leaves. But every leaf
of the $(C(RF)\to j)$ tree can similarly be shown to have at least
two children in the $(A(RF)\to i)$ trees, implying
that the $(A(RF)\to i)$ trees have more than $r$ leaves, a contradiction.
\end{proof}

\begin{cor}
\label{cor:BicoherentImpliesCycle}
Let $T$ be bicoherent with $[T]$ a nontrivial $V_Q$-type, such that
$\widehat{T}$ is also nontrivial, and such that
$G_P$ has no new $x$-cycle on $A(T)$ or $C(T)$. Then there is a new
$x$-cycle on $A(\widehat{T})$ or $C([\widehat{T}])$.
\end{cor}
\begin{proof}
 Taking the conjugate of a type swaps $B_s$ with $A$ and $B_{-s}$ with $C$ 
 for some choice of sign $s$.
\end{proof}

%% file: sections/CablingConjecture.tex
\section{The Cabling Conjecture}

\begin{thm}
\label{thm:CCFiveBridgeKnots}
The cabling conjecture is true for $b$-bridge knots with $b\leq 5$.
\end{thm}

\begin{proof}
Recall that by Theorem \ref{thm:AllTypesTorsion}, if $G_P$ represents all
$V_Q$-types, then $H_1(M(\gamma))$ has torsion. But $M(\gamma)$ is $S^3$,
so $G_P$ cannot represent all $V_Q$-types. We will show that if a $V_Q$-type
$\tau$ exists which $G_P$ does not represent, then $q > 2v + 2$, where
$v$ is the number of vertices in the smallest great web in $G_Q$.
By Corollary \ref{cor:MinWebSize}, $v\geq 4$, so $q > 10$.

If $G_P$ does not represent the $V_Q$-type $[T] = \tau$, then $\tau$ falls
in one of the following cases:
\begin{enumerate}
 \item \label{itm:trivType} $\tau$ is the trivial $V_Q$-type;
 \item \label{itm:coherentNontriv} $\tau$ is coherent and nontrivial,
 and no cycle on $A(T)$ ($C(T)$) bounds two disks each containing
 vertices which are not in $A(T)$ ($C(T)$);
 \item \label{itm:incoherent} $\tau$ is incoherent with sequence of
 coherence $(T_1,\ldots,T_n)$, and no cycle on $A(T_n)$
 ($C(T_n)$) bounds two disks each containing vertices which are not in
 $A(T_n)$ ($C(T_n)$);
 \item \label{itm:TwoBases} $\tau$ is nontrivial with
 sequence of coherence $(T_1,\ldots,T_n)$, and there
 exists a cycle on $A(T_n)$ ($C(T_n)$) which bounds two disks both
 containing vertices not in $A(T_n)$ ($C(T_n)$).
\end{enumerate}

\begin{thm}
 \label{thm:trivTypeRepped}
 $G_P$ represents the trivial $V_Q$-type.
\end{thm}
\begin{proof}
 Suppose $G_P$ does not represent the trivial $V_Q$-type.
 By Lemma \ref{lem:AntiparallelEdgesInQGiveCircuitInP}, if $G_Q$ contains
 more than $p - 2$ edges between
 antiparallel vertices, then $G_P$ has a cycle on parallel vertices.
 $G_P$ has $\mathbb{P}(V_r, V_Q)$ by Lemma \ref{lem:NoIsolatedVertices},
 so by Lemma \ref{lem:CircuitImpliesRepsTrivialType}, $G_P$ cannot have a
 cycle on parallel vertices.
 Thus $G_Q$ contains at most $p - 2$ edges between antiparallel vertices.
 
 $G_Q$ contains at least two innermost $(s)$-sets, say $V$ and $W$.
 By Lemma \ref{lem:TrivialInduction}, $G_P(V)$ does not represent the
 trivial $V$-type and $G_P(W)$ does not represent the trivial $W$-type.
 Theorem \ref{thm:TrivialTypeSDisks}(\ref{itm:NEdgesLeave}) shows that
 precisely $p-2$ edges leave $V$ and precisely $p-2$ edges leave $W$.
 But since $V$ (respectively $W$) is an $(s)$-set, any edge leaving $V$
 (respectively $W$) is an edge between antiparallel vertices.
 Thus $V$ must contain all (we may assume) positive
 vertices and $W$ must contain all negative vertices.
 By Theorem \ref{thm:TrivialTypeSDisks}(\ref{itm:RegVertsLeave}),
 $V^* = W^* = V_r$.
 This implies that there can be no edge between $V$ and $W$ which meets
 the label $1$ at either end, a contradiction with
 Theorem \ref{thm:TrivialTypeSDisks}(\ref{itm:TreeExistence}).
 Thus $G_P$ must represent the trivial $V_Q$-type.
\end{proof}


Note that by Corollary \ref{cor:BicoherentImpliesCycle}, if $\tau$ is in
case \ref{itm:coherentNontriv}, then $\widehat{\tau}$ is in case
\ref{itm:incoherent} or \ref{itm:TwoBases}.
Since $G_P$ does not represent $\tau$ if and only if $G_P$ does not
represent $\widehat{\tau}$, it is sufficient to consider cases
\ref{itm:incoherent} and \ref{itm:TwoBases}.

Case \ref{itm:incoherent}:
Suppose $\tau$ is incoherent with sequence of coherence $(T_1,\ldots,T_n)$,
with no cycle on $A(T_n)$ ($C(T_n)$) bounding two disks each containing
vertices which are not in $A(T_n)$ ($C(T_n)$). Since $T_n$ is the first
coherent type in the sequence, $n > 1$ and $T_n$ is nontrivial.
By Lemma \ref{lem:HoffmanLemma}, $A(T_n)$ and $C(T_n)$ are great webs,
so $\vert L(T_n) \vert \geq 2v$.
But notice that by the definition of derivative type,
either $A(T_n)\cup C(T_n)\subset A(T)$ or $A(T_n)\cup C(T_n)\subset C(T)$.
Thus $\vert A(T) \vert = \vert C(T) \vert \geq 2v$. This immediately
implies that $q \geq 4v$. In fact, if $q = 4v$, then
$\vert A(T) \vert = \vert C(T) \vert = 2v$, so $\widehat{\tau}$ is the trivial
$V_Q$-type. Theorem \ref{thm:trivTypeRepped} implies that $G_P$ represents
$\widehat{\tau}$ and therefore $\tau$, contrary to our assumptions.
Thus $q \geq 4v + 2$.

Case \ref{itm:TwoBases}:
Suppose $\tau$ is nontrivial and has a sequence of coherence
$(T_1,\ldots,T_n)$ and a cycle $\sigma$ on $A(T_n)$ ($C(T_n)$) such that
$\sigma$ bounds two disks each containing vertices other than $A(T_n)$
($C(T_n)$).
By Corollary \ref{cor:GL2.4.2}, $G_P(L(T_n))$ does not represent $[T_n]$.
Let $D_1$ and $D_2$ be the two $\sigma$-disks, with $\sigma$-sets $V_1$ and
$V_2$. By Lemma \ref{lem:NontrivialInduction},
$G_P(V_1)$ does not represent some $V_1$-type
$\tau_1$ and $G_P(V_2)$ does not represent some $V_2$-type $\tau_2$.


Consider just $D^1 = D_1$ and $V^1 = V_1$, and let $\sigma^1 = \sigma$. 
For $[T^i] = \tau^i$ a $V^i$-type with sequence of coherence
$(T^i_1,\ldots,T^i_n)$, we are interested in defining $\sigma^{i+1}$ as follows,
if possible (i.e. if such cycles exist):
\begin{enumerate}
 \item a new $x$-cycle on $A(T^i_n)$ or $C(T^i_n)$ if $\tau^i$ is nontrivial; or
 \item a cycle on $A(T^i_n)$ ($C(T^i_n)$) such that $\sigma^{i + 1}$ bounds two
 disks each containing vertices other than $A(T_n)$ ($C(T_n)$), if $\tau^i$
 is nontrivial and no new $x$-cycle exists on $A(T^i_n)$ or $C(T^i_n)$; or
 \item a new $x$-cycle, if $\tau^i$ is trivial.
\end{enumerate}
When we find such a $\sigma^{i+1}$, let $D^{i+1}$ be the $\sigma^{i+1}$-disk
contained in $D_1$, and let $V^{i + 1}$ be the corresponding $\sigma^{i+1}$-set.
By Lemmas \ref{lem:TrivialInduction} and \ref{lem:NontrivialInduction},
there is some $V^{i + 1}$-type $\tau^{i + 1}$ such that $G_P(V^{i + 1})$
does not represent $\tau^{i + 1}$.
By finiteness we can find $T^\alpha$, $D^\alpha\subset D_1$,
$\sigma^\alpha$, and $V^\alpha\subset V_1$
such that no qualifying $\sigma^{\alpha+1}$ cycle exists.
In the same way, we can find $T^\beta$, $D^\beta\subset D_2$,
$\sigma^\beta$, and $V^\beta\subset V_2$
such that no qualifying $\sigma^{\beta+1}$ cycle exists.


\begin{lem}
 \label{lem:nontrivInductionDisks}
 $D^\alpha$ and $D^\beta$ are nontrivial $\sigma^\alpha$ and $\sigma^\beta$
 disks, respectively.
\end{lem}
\begin{proof}
 For $\gamma\in \{\alpha, \beta\}$, if $\sigma^\gamma$ is a new $x$-cycle, the result follows immediately from
 Theorem \ref{thm:NoNewGreatXCycle}.
 The only way $\sigma^\gamma$ is not a new $x$-cycle is if $\tau^{\gamma - 1}$
 is nontrivial and there is no new $x$-cycle on
 $A(T^{\gamma - 1}_{n})$ or $C(T^{\gamma - 1}_{n})$.
 Thus the second part of Lemma \ref{lem:HoffmanLemma} applies, and since
 $\sigma^\gamma$ is on vertices of either $A(T^{\gamma - 1}_{n})$ or
 $C(T^{\gamma - 1}_{n})$ (without loss of generality, suppose
 $A(T^{\gamma - 1}_{n})$), at most $p - 2$ edges leave $V^\gamma$.
 
 If $D^\gamma$ is trivial, then every $V^\gamma$ vertex has the same sign.
 By Corollary \ref{cor:NoXEdgesLeavingImpliesXCycle} and Theorem
 \ref{thm:NoNewGreatXCycle}, exactly $p - 2$ edges
 leave $V^\gamma$, one at each label in $V_r$. Thus every edge leaving
 $A(T^{\gamma - 1}_{n})$ goes to a label in $V_r$, a contradiction of
 Theorem \ref{thm:HoffmanTree}.
\end{proof}

We break case \ref{itm:TwoBases} from above
into the following three subcases:
 \begin{enumerate}[(i)]
  \item \label{itm:ZeroTriv} $[T^\alpha_n],[T^\beta_n]$ are both nontrivial;
  \item \label{itm:OneTriv} Exactly one of $[T^\alpha_n],[T^\beta_n]$ is
  nontrivial;
  \item \label{itm:TwoTriv} $[T^\alpha_n],[T^\beta_n]$ are both trivial.
 \end{enumerate}
 
 Case \ref{itm:ZeroTriv}:
 
 Part 2 of Lemma \ref{lem:HoffmanLemma} can be applied to both disks,
 implying the existence of at least $2$ great webs in each.
 This implies that $\vert V^\alpha \vert \geq 2v$ and
 $\vert V^\beta \vert \geq 2v$.
 Since $\Sigma$ is disjoint from both, we have $q > 4v$.

 Case \ref{itm:OneTriv}:
 
 Suppose that $\tau^\alpha_n$ is nontrivial and $\tau^\beta_n$ is trivial.
 Then as in Case \ref{itm:ZeroTriv}, $\vert V^\alpha \vert \geq 2v$.
 Meanwhile, Lemma \ref{lem:nontrivInductionDisks} implies that $V^\beta$
 is antiparallel to $\sigma^\beta$. Thus by Theorem
 \ref{thm:TrivialTypeSDisks}, $V^\beta$ is a great web. Therefore
 by Corollary \ref{cor:MinWebSize}, $\vert V^\beta \vert \geq v$.
 Since $T^\alpha_n$ is coherent, either both $A(T^\alpha_n)$ and
 $C(T^\alpha_n)$ are the same sign, or one is the same sign as $V^\beta$.
 Therefore at least $2v$ vertices must exist of each sign, so $q \geq 4v$.
 
 Case \ref{itm:TwoTriv}:
 
 Suppose both $\tau^\alpha_n$ and $\tau^\beta_n$ are trivial. As in Case
 \ref{itm:OneTriv}, $\vert V^\alpha \vert \geq v$ and
 $\vert V^\beta \vert \geq v$. Since $\sigma$ is disjoint from both,
 $q > 2v$.
 
 Suppose $q = 2v + 2$. Then there is an antiparallel vertex $x^\alpha$ to
 $V^\alpha$ with a loop bounding a disk containing only $V^\alpha$.
 By Theorem \ref{thm:TrivialTypeSDisks}, $G([V^\alpha,\{x^\alpha\}])$
 consists entirely of a negative $(V^\alpha \to \{x^\alpha\}\to 1)$ tree
 and a positive $(V^\alpha \to \{x^\alpha\}\to 2)$ tree.
 Note that the edges between the special vertices $1$ and $2$ split
 $G_P$ into several disks $\Delta_i$. By Lemma 
 \ref{lem:DeltaCornersHaveDisjointLabels}, the negative 
 $(V^\alpha \to \{x^\alpha\}\to 1)$ and the positive
 $(V^\alpha \to \{x^\alpha\}\to 2)$ must be in distinct $\Delta_i$.
 Hence there can be no edges between parallel regular vertices
 in $G_Q$, contradicting the existence of the great web $V^\alpha$.
 Thus $q > 2v + 2$.
 
We have shown that if $G_P$ does not represent a $V_Q$-type $\tau$,
$q > 2v + 2$. By Corollary \ref{cor:MinWebSize}, $v\geq 4$, so
$q > 10$.
Thus either $G_P$ represents all types or the bridge number
$b > 5$. By Theorem \ref{thm:AllTypesTorsion}, $G_P$ cannot represent
all types, hence $b > 5$.
\end{proof}

\begin{rem}
 Note that $q = 12$ is only possible if the type $\tau$ which is not represented
 by $G_P$ falls in case \ref{itm:TwoBases}(\ref{itm:TwoTriv}), and 
 $b = 6$ is only possible if $q = 12$ and every thin presentation of $k$
 is also bridge position.
 Similarly, $q = 14$ (and hence $b = 7$) is only possible if $\tau$ falls
 in case \ref{itm:TwoBases}(\ref{itm:TwoTriv}).
\end{rem}